\documentclass{article}

\usepackage{arxiv}

\usepackage[utf8]{inputenc} % allow utf-8 input
\usepackage[T1]{fontenc}    % use 8-bit T1 fonts
\usepackage{hyperref}       % hyperlinks
\usepackage{url}            % simple URL typesetting
\usepackage{booktabs}       % professional-quality tables
\usepackage{amsfonts}       % blackboard math symbols
\usepackage{nicefrac}       % compact symbols for 1/2, etc.
\usepackage{microtype}      % microtypography
\usepackage{graphicx}
\usepackage{doi}
\usepackage{amsthm}
\usepackage{dsfont}
\usepackage{algpseudocode}
\usepackage{algorithm}
\usepackage{tikz}
\newcommand{\tikzcircle}[2][black,fill=black]{\tikz[baseline=-0.5ex]\draw[#1,radius=#2] (0,0) circle ;}%
\usepackage{mathtools}
\usepackage{tcolorbox}
\usepackage{amsmath}
\usepackage{nicefrac}
\usepackage{tikz-qtree}
\usepackage{forest}
\usepackage{physics}
\usepackage{bbm}
\usepackage{wasysym}

\newtheorem{theorem}{Theorem}
\newtheorem{lemma}{Lemma}
\newtheorem{proposition}{Proposition}

\theoremstyle{definition}
\newtheorem{definition}{Definition}

\newtheorem{remark}{Remark}

\DeclareMathOperator{\PP}{P}
\newcommand{\mi}{\mathrm{i}}

\newcommand{\C}{\mathbb{C}}
\newcommand{\R}{\mathbb{R}}
\newcommand{\B}{\mathcal{B}}
\newcommand{\cL}{\mathcal{L}}
\newcommand{\K}{\mathcal{K}}

\newcommand{\D}{\mathbf{D}}
\newcommand{\G}{\mathbf{G}}
\renewcommand{\L}{\mathbf{L}}

\title{Generalized convolution quadrature based on the trapezoidal rule}

\author{{\small	Lehel Banjai} \\
{\small	Maxwell Institute for Mathematical Sciences} \\
{\small Department of Mathematics}\\ 
{\small	Heriot-Watt University} \\
{\small	Edinburgh, EH14 4AS, United Kingdom} \\
{\small	\texttt{l.banjai@hw.ac.uk}}
\And
{\small	Matteo Ferrari} \\
{\small	Dipartimento di Scienze Matematiche ``G.L. Lagrange'' }\\ 
{\small	Politecnico di Torino} \\
{\small	Torino, 10129, Italy} \\
{\small	\texttt{matteo.ferrari@polito.it}}
}

\renewcommand{\headeright}{}
\renewcommand{\undertitle}{}
\renewcommand{\shorttitle}{}

\hypersetup{
pdftitle={Trapezoidal_gCQ},
pdfsubject={q-bio.NC, q-bio.QM},
pdfauthor={Ferrari M.~Banjai L.},
pdfkeywords={Convolution quadrature, trapezoidal rule, non-uniform time stepping},
}

\begin{document}
\maketitle

\begin{abstract}
We present a novel generalized convolution quadrature method that accurately approximates convolution integrals. During the late 1980s, Lubich introduced convolution quadrature techniques, which have now emerged as a prevalent methodology in this field. However, these techniques were limited to constant time stepping, and only in the last decade generalized convolution quadrature based on the implicit Euler and Runge-Kutta methods have been developed, allowing for variable time stepping. In this paper, we introduce and analyze a new generalized convolution quadrature method based on the trapezoidal rule. Crucial for the analysis is the connection to a new modified divided difference formula that we 
establish. Numerical experiments demonstrate the effectiveness of our method in achieving highly accurate and reliable results.
\end{abstract}


\section{Introduction}
Convolution operators are widely used in applications that involve linear time-invariant non-homogeneous evolution equations, including wave and heat propagation problems, and occur in integral equations, such as Volterra, and Wiener-Hopf equations. In this paper, we present a numerical method for computing or solving linear convolution equations of the form
\begin{equation} \label{equation1} 
	\int_0^t \kappa(t-\tau) g(\tau) \dd \tau = \phi(t), \quad t \ge 0,
\end{equation}
where $\kappa$ is a fixed kernel operator and $g$ (or $\phi$) is a given function. In many applications, the Laplace transform $\K$ of the convolution kernel $\kappa$ is known or easier to evaluate than $\kappa$. The Convolution Quadrature (CQ) method involves expressing $\kappa$ as the inverse Laplace transform of a transfer operator $\K$, formulating the problem as an integro-differential equation in the Laplace domain, and approximating the differential equation using a time-stepping method such as linear multisteps \cite{Lubich1988a, Lubich1988b, Lubich1994, Lubich2004} or Runge-Kutta \cite{LubichOstermann1993, BanjaiLubich2011, BanjaiLubichMelenk2011, BanjaiFerrari2022}. The resulting discrete convolution equation can then be solved numerically. 

The original CQ method is strongly restricted to fixed time step integration. However, in recent works \cite{LopezFernandezSauter2013, LopezFernandezSauter2015a, LopezFernandezSauter2016} the generalized Convolution Quadrature (gCQ) has been introduced with variable time stepping, enabling adaptive resolution of non-smooth temporal behaviours. Moreover, utilizing non-uniform time stepping schemes can facilitate progress towards adaptive time stepping for parabolic and hyperbolic evolution equations. The first approach was limited to first-order implicit Euler scheme \cite{LopezFernandezSauter2013, LopezFernandezSauter2015a}, and was later extended to Runge-Kutta methods in \cite{LopezFernandezSauter2016}. Applications of gCQ have been demonstrated in various fields, including acoustics with absorbing boundary conditions \cite{SauterSchanz2017}, uncoupled quasistatic thermoelasticity in \cite{LeitnerSchanz2021}, and approximation of fractional integrals and associated fractional diffusion equations \cite{JingLopezFernandez2022}. In \cite{LopezFernandezSauter2013}, gCQ was introduced and formulated via high order divided differences of the transfer operator $\K$, which was appropriate for the stability and error analysis, but less suited for efficient algorithmic realization. However, in \cite{LopezFernandezSauter2015a}, an efficient algorithmic formulation of gCQ was presented. It is based on the approximation of divided differences by quadrature in the complex plane, following the approach proposed in \cite{LopezFernandezSauter2015b}. This new formulation allows for faster and more efficient computation of gCQ.

The original analysis by Lubich \cite{Lubich1988a} excluded CQ based on the trapezoidal rule method for technical reasons. However, it was known that the trapezoidal-based method outperforms the first-order backward Euler method and BDF2, which is too dispersive. In the appendix of \cite{Banjai2010} an initial analysis was developed for the CQ based on the trapezoidal rule, which was further refined in \cite{ErusluSayas2020}. The goal of this paper is to introduce and analyze the trapezoidal gCQ. This method results in much faster convergence rates and improved long time behaviour compared to the implicit Euler method.

The paper is organized as follows: in Section \ref{sec2} we provide a brief overview of one-sided convolution operators and introduce the class of convolution kernels that we consider in this paper. Section \ref{sec3} presents the trapezoidal gCQ which is a method for discretizing convolution operators using variable time stepping. In Section \ref{sec4}, we analyze the stability and convergence of the method and derive a Leibniz formula for a new divided differences rule which is related to the gCQ weights. Section \ref{sec5} presents an algorithm for the practical realization of the trapezoidal gCQ. The algorithm is based on a contour integral representation of the numerical solution and quadrature in the complex plane. We conclude with numerical experiments to demonstrate that the trapezoidal gCQ converges with optimal convergence rates for problems where the regularity of the solution is not uniformly distributed in the time interval, while other CQ-type methods converge suboptimally. Additionally, we present numerical examples for gCQ based on BDF2, although we have not yet developed a theoretical analysis for this case.

%%%%%%%%%%%%%%%%%%%%%%%%%%%%%%%%%%%%%%%%%%%%%%
\section{Convolution quadrature for hyperbolic symbols} \label{sec2}
We consider the class of convolution operators as described in \cite[Section 2.1]{Lubich1988a} (see also \cite[Section 2]{BanjaiSayas2022}).

Let $X$ and $Y$ denote two normed vector spaces, and let $\B(X,Y)$ be the space of continuous, linear mappings from $X$ to $Y$. As a norm in $\B(X,Y)$ we consider the operator norm
\begin{equation*}
	\| \K \|_{\B(X,Y)} := \sup_{g \in X \setminus \{0\}} \frac{ \| \K g \|_Y}{\| g \|_X}.
\end{equation*}
Let define also the spaces $\C_+ := \{ s \in \C : \Re s > 0\}$, and $\C_{\sigma_0} := \{ s \in \C : \Re s > \sigma_0 \}$ for some $\sigma_0 >0$.

We are interested in the one-sided convolution
\begin{equation}
	\int_0^t \kappa(t-\tau) g(\tau) \dd \tau, \quad  t \ge 0,
\end{equation}
of causal ($f(t)=0, t<0$) distributions $\kappa$ and $g$. The kernel operator $\kappa$ is the inverse Laplace transform of some transfer operator $\K : \C_+ \to \B(X,Y)$, which is assumed to be an \textit{hyperbolic symbol}.
\begin{definition}[Hyperbolic Symbol]
For given normed vector spaces $X,Y$ and $\mu \in \R$, the space of \textit{hyperbolic symbols} $\mathcal{A}(\mu, \B(X,Y))$ is the space of functions $\K : \C_+ \to \B(X,Y)$ analytic in $\C_+$ and satisfying 
\begin{equation} \label{assump}
	\| \K(s) \|_{\mathcal{B}(X,Y)} \le M | s |^\mu, \quad s \in \C_{\sigma_0},
\end{equation}
for some $\sigma_0>0$ and $M > 0$.
\end{definition}
If $\mu < -1$, the time-domain operator $\kappa := \cL^{-1}\{\K\}$ is well-defined by the Bromwich integral
\begin{equation} \label{Bromw}
	\kappa(t) := \cL^{-1} \{\K\}(t) = \frac{1}{2 \pi \mi} \int_{\sigma + \mi \R} e^{st} \K(s) \dd s
\end{equation}
for $\sigma > \sigma_0$ and $\sigma_0$ as in \eqref{assump}.

If $\mu \ge -1$, we let the integer $\rho := \lfloor \mu \rfloor+1$ and $\K_\rho(s) := s^{-\rho}\K(s)$. Let $\kappa_\rho := \cL^{-1}\{\K_\rho\}$, where again the inverse Laplace transform is defined by  the Bromwich integral \eqref{Bromw}. We see that $\cL\{\kappa\} = \K$ where $\kappa :=  \partial_t^\rho \kappa_\rho$, and $\partial_t^\rho$ is the casual distributional derivative (see e.g. \cite{BanjaiSayas2022}). We are now able to define the convolution for $\mu \ge -1$ by
\begin{equation} \label{convolution}
	\K(\partial_t) g(t) := \frac{\partial^\rho}{\partial t^\rho} \int_0^t \kappa(t-\tau) g(\tau)  \dd \tau = \int_0^t \kappa_\rho(t-\tau) g^{(\rho)}(\tau)  \dd \tau, \quad  t \ge 0,
\end{equation}
for casual functions $g \in C^{\rho-1}(\mathbb{R})$ satisfying $g^{(j)}(0) =0 $, $j = 0,\dots,\rho-1$, and $g^{(\rho)}$ locally integrable. If $g$ is only defined on a finite interval $[0, T]$, we can extend it by the Taylor polynomial 
\begin{equation*}
	g(t) := \sum_{j=0}^\rho \frac{1}{j!} g^{(j)}(T)(t-T)^j, \quad t > T
\end{equation*}
and define $\K(\partial_t)g$ as above.

The motivation behind  the operational notation  $\K(\partial_t)g$, can be seen when considering the case $\K(s) = s$, where the above definition implies that $\K(\partial_t)g = \partial_t g$. Furthermore, the composition rule $\K_2\K_1(\partial_t)g = \K_2(\partial_t)\K_1(\partial_t)g$ holds for hyperbolic symbols $\K_1$ and $\K_2$. 

Convolution quadrature (CQ) is a discretization of one-sided convolutions $\K(\partial t)g$ for hyperbolic symbols based on particular ODE-solvers. Even if in literature there are various choices of high-order CQ based on Runge-Kutta methods (see e.g \cite{BanjaiLubichMelenk2011, BanjaiFerrari2022}), we focus here on CQ based on A-stable linear multistep methods (see \cite{Lubich1988a, Lubich1988b}), and thus restricted by the Dahlquist's barrier to second order methods. 

Given a fixed time-step $\Delta > 0$, the CQ is defined by the discrete convolution
\begin{equation} \label{unifCQ}
	\K\left(\partial_t^{\Delta}\right)g(t_n) := \sum_{j=1}^n \omega_{n-j}(\K_\rho) g^{(\rho)}(t_j)
\end{equation}
where $t_j := j \Delta$. The convolution weights $\omega_j(\K_\rho)$ are expressed by the contour integral representation
\begin{equation} \label{weightsUNIF}
	\omega_j(\K_\rho) := \frac{1}{2 \pi \mi} \oint_\mathcal{D} \K_\rho\left( \frac{\delta(s)}{\Delta} \right) s^{-j-1} \dd s,
\end{equation}
where $\delta(\zeta)$  is a generating function of an A-stable linear multistep method, and $\mathcal{D}$ is a proper complex contour. A standard choice is $\mathcal{D}$ a circle of radius $0<\lambda<1$ that leads to the approximations via the compound trapezoidal rule
\begin{equation*}
	\omega_j(\K_\rho) \approx \frac{\lambda^{-j}}{L+1} \sum_{\ell=0}^L \K_\rho \left( \frac{\delta(\lambda e^{- \ell \frac{2\pi \mi}{L+1}})}{\Delta} \right) e^{\ell j \frac{2\pi \mi}{L+1}}
\end{equation*}
efficiently computable for all $j=0,\ldots,L$ simultaneously via the Fast Fourier Transform.

The CQ method as described above and its standard analysis heavily depend on the use of constant time stepping. However, in the next section, we will present a potential extension of this method to non-uniform time stepping schemes.

%%%%%%%%%%%%%%%%%%%%%%%%%%%%%%%%%%%%%%%%%%%%%%%
\section{Generalized convolution quadrature based on the trapezoidal rule} \label{sec3}

In order to expand upon the gCQ based on the backward Euler scheme outlined in \cite{LopezFernandezSauter2013} , we introduce the gCQ derived from the trapezoidal rule.

By applying the inverse Laplace transform to $\kappa_\rho$ via the Bromwich representation \eqref{Bromw}, we can write \eqref{convolution} as
\begin{equation*}
	\K(\partial_t) g(t) = \int_0^t \left( \frac{1}{2\pi \mi} \int_{\sigma + \mi \mathbb{R}} e^{s(t-\tau)} \K_\rho(s) \dd s \right) g^{(\rho)}(\tau) \dd \tau, \quad t \ge 0
\end{equation*}
and interchanging the order of integration, we readily obtain
\begin{equation} \label{convo}
	\K(\partial_t) g(t) = \frac{1}{2\pi \mi} \int_{\sigma + \mi \mathbb{R}} \K_\rho(s) u(t;s) \dd s, \quad t \ge 0,
\end{equation}
where
\begin{equation*}
	u(t;s) := \int_0^t e^{s(t-\tau)} g^{(\rho)}(\tau) \dd \tau.
\end{equation*}
Note that $u(t;s)$ is the unique causal solution of following the simple initial value problem
\begin{equation} \label{ode}
	\begin{cases}
		\partial_t u(t;s) = s u(t;s) + g^{(\rho)}(t), \\
		u(0;s) = 0.
	\end{cases}
\end{equation}
In the case of uniform CQ \eqref{unifCQ}, the key point now is to consider the values of $\K(\partial_t)g$ at a finite number of equidistant abscissas $t_n$ and to replace in \eqref{convo} the functions $u(t_n; s)$ by an approximation of them, that we obtain by applying to \eqref{ode} a linear multistep ODE solver having proper stability properties. We aim, instead, to discretize \eqref{ode} with the trapezoidal rule associated to a non-uniform time mesh.

Given $0 = t_0 < t_1 < \ldots < t_N = T$ with non-uniform time-steps $\Delta_n := t_n - t_{n-1}, n = 1, \ldots , N$, the trapezoidal rule when used to approximate the solution of the initial value problem \eqref{ode}, results in the following difference equation:
\begin{equation*}
u_n(s) = u_{n-1}(s) +  \frac{1}{2} \Delta_n \left(su_{n-1}(s) + g^{(\rho)}(t_{n-1}) + s u_n(s) + g^{(\rho)}(t_n)\right)
\end{equation*}
where $u_n(s) \approx u(t_n;s)$, for $n = 1,\ldots,N$, and $u_0(s)=0$. Solving for $u_n(s)$ leads to
\begin{equation} \label{solTrap}
	u(t_n;s) \approx u_n(s) = u_{n-1}(s)  \frac{2+\Delta_n s}{2- \Delta_n s} +  \bigl(g^{(\rho)}(t_{n-1}) + g^{(\rho)}(t_n)\bigr) \frac{\Delta_n}{2-\Delta_n s}, \quad n = 1,\ldots,N.
\end{equation}
The recursion can be iteratively solved to obtain the following expression
\begin{align} \label{recursion}
u_n(s)  = \sum_{j=1}^{n} g^{(\rho)}(t_j) D_j^n \prod_{k=j+2}^n \left(2\Delta_k^{-1}+s\right) \prod_{k=j}^n \left(2\Delta_k^{-1}- s\right)^{-1}
\end{align}
where the coefficients $D_j^n$ are defined as follows
\begin{equation} \label{djn}
	D_j^n := 
	\begin{cases}
		2\left(\Delta_j^{-1} + \Delta_{j+1}^{-1}\right) & j < n, \\
		1 & j = n.
	\end{cases}
\end{equation}
By considering \eqref{convo} at the time point $t_n$ and substituting $u(t_n;s)$ by the approximation $u_n(s)$ in \eqref{recursion}, we obtain the non-uniform approximation of the convolution $\K(\partial_t)g$
\begin{align*}
	\K\left(\partial_t^{\{\Delta_j\}}\right)g(t_n) & := \frac{1}{2 \pi \mi} \int_{\sigma + \mi \mathbb{R}} \K_\rho(s) u_n(s) \dd s 
	\\ & = \sum_{j=1}^n g^{(\rho)}(t_j) D_j^n \frac{1}{2 \pi \mi} \int_{\sigma + \mi \mathbb{R}} \K_\rho(s) \prod_{k=j+2}^n \left(2\Delta_k^{-1}+ s\right) \prod_{k=j}^n \left(2\Delta_k^{-1}- s\right)^{-1} \dd s.
\end{align*}
We simplify the latter expression by writing
\begin{align} \label{gCQ}
	\K\left(\partial_t^{\{\Delta_j\}}\right)g(t_n) = \sum_{j=1}^n w_{n,j}(\K_\rho) g^{(\rho)}(t_j)
\end{align}
where we have defined the weights as
\begin{equation} \label{weights}
	w_{n,j}(\K_\rho) := D_j^n \frac{1}{2 \pi \mi} \oint_{\mathcal{C}} \K_\rho(s) G_j^n(s) \dd s, \qquad \text{with~} \quad G_j^n(s) := \prod_{k=j+2}^n \left(2\Delta_k^{-1}+ s\right) \prod_{k=j}^n \left(2\Delta_k^{-1}- s\right)^{-1}
\end{equation}
and $\mathcal{C}$ is a negatively oriented contour contained in the right half complex plane surrounding all the $N$ poles $2\Delta_k^{-1}$.
\begin{remark}
 Integrating over the contour $\mathcal{C}$ and integrating along the line $\sigma+ \mi \R$ both produce the same result. However, by choosing a suitable contour $\mathcal{C}$, more efficient quadrature techniques can be employed to calculate the weights. This approach has been extensively demonstrated and substantiated in \cite{LopezFernandezSauter2015a}, resulting in improved computational performance and accuracy. We will revisit these quadrature rules in Section \ref{sec5} for further clarification.
\end{remark}
Our initial step is to establish that the trapezoidal-based gCQ, analogous to the method outlined in \cite[Remark 2.29]{BanjaiSayas2022} for backward Euler-based gCQ, simplifies to the standard CQ when time-steps are uniform.
\begin{proposition}
Let $\K \in \mathcal{A}(\mu,\B(X,Y))$ for some $\mu \in \R$ and let $\rho = \lfloor\mu\rfloor+1$. Let $0<t_0<t_1<\ldots<t_N=T$ be the discrete times with uniform time steps $\Delta = t_{j+1}-t_j$. Then, the trapezoidal based gCQ \eqref{gCQ} coincides with the standard trapezoidal based CQ \eqref{unifCQ}, i.e., we have
\begin{equation*}
	\omega_{n,j}(\K_\rho) = \omega_{n-j}(\K_\rho), \quad \text{for all~~~} 0 < j \le n \le N.
\end{equation*}
\end{proposition}
\begin{proof}
The generating function of the trapezoidal rule is $\delta(\zeta) = 2 \frac{1-\zeta}{1+\zeta}$. In the uniform case, the weights are expressed as given in \eqref{weightsUNIF} by
\begin{equation} \label{wnUni}
	\omega_{n-j}(\K_\rho) = \frac{1}{2 \pi \mi} \oint_\mathcal{D} \K_\rho\left( \frac{2}{\Delta} \frac{1-s}{1+s} \right) s^{-(n-j)-1} \dd s,
\end{equation}
where $\mathcal{D}$ a circle of fixed radius $0<\lambda<1$.
When we have equal time-steps $\Delta = \Delta_1 = \ldots = \Delta_N$, the gCQ weights defined in \eqref{weights} can be written as
\begin{equation} \label{eq13}
	\begin{aligned} 
		w_{n,j}(\K_\rho) & = D_j^n \frac{1}{2 \pi \mi} \oint_{\mathcal{C}} \K_\rho(s) G_j^n(s) \dd s
		\\ & = D_j^n \frac{1}{2 \pi \mi} \oint_{\mathcal{C}} \K_\rho(s) \prod_{k=j+2}^n \left(2\Delta^{-1}+ s\right) \prod_{k=j}^n \left(2\Delta^{-1}-s\right)^{-1} \dd s
		\\ & = D_j^n \frac{1}{2 \pi \mi} \oint_{\mathcal{C}} \K_\rho(s) \left(2\Delta^{-1}+ s\right)^{\max\{0,n-j-1\}} \left(2\Delta^{-1}-s\right)^{-(n-j+1)} \dd s.
	\end{aligned}
\end{equation}
Here, $\mathcal{C}$ is a complex contour located in the right half-plane and encircling the singularity $s = 2\Delta^{-1}$.

Referring to \eqref{djn}, we can distinguish between the cases when $j=n$ when $j < n$. In the former case, we deduce
\begin{align*}
	w_{n,n}(\K_\rho) = \frac{1}{2 \pi \mi} \oint_{\mathcal{C}} \K_\rho(s) \left(2\Delta^{-1}-s\right)^{-1} \dd s,
\end{align*}
while in the latter case, we obtain
\begin{equation} \label{unifo}
	w_{n,j}(\K_\rho) = 4\Delta^{-1} \frac{1}{2 \pi \mi} \oint_{\mathcal{C}} \K_\rho(s) \left(2\Delta^{-1}+ s\right)^{n-j-1} \left(2\Delta^{-1}-s\right)^{-(n-j+1)} \dd s.
\end{equation}
We see that $\omega_{n,n}(\K_\rho) = \omega_0(\K_\rho) = f(0)$ where $f(\zeta) = \K_\rho\left(\frac{2}{\Delta} \frac{1-\zeta}{1+\zeta}\right)$ by the Cauchy's integral Theorem.

For $j < n$, we can use the Moebius map $\phi(z) := \frac{2}{\Delta} \frac{(1-z)}{(1+z)}$ to make the change of variables in \eqref{unifo}. 
Specifically, we set $s = \frac{2}{\Delta} \frac{(1-\zeta)}{(1+\zeta)}$, from which we can deduce that 
\begin{equation*}
	\dd s = -\frac{4}{\Delta} \frac{1}{(1+\zeta)^2} \dd \zeta, \quad \left(2\Delta^{-1} + s\right) = \frac{4}{\Delta} \frac{1}{1+\zeta} \quad \text{~and~} \quad \left(2\Delta^{-1}-s\right) = \frac{2}{\Delta} \frac{\zeta^2}{1+\zeta}.\end{equation*} 
Using these substitutions, we can write 
\begin{equation} \label{wnjUni}
	w_{n,j}(\K_\rho) = \frac{1}{2 \pi \mi} \oint_{\phi^{-1}(\mathcal{C})} \K_\rho\left(\frac{2}{\Delta} \frac{1-\zeta}{1+\zeta}\right) \zeta^{-(n-j)-1} \dd \zeta.
\end{equation}
Consider choosing $\mathcal{C}$ in \eqref{eq13} to be the circle centered at $\left(\frac{2}{\Delta} \frac{(1+\lambda^2)}{(1-\lambda^2)}, 0\right)$ with radius $\frac{2}{\Delta} \sqrt{ \frac{(1+\lambda^2)^2}{(1-\lambda^2)^2} - 1}$ for a fixed $0 < \lambda < 1$. This circle includes the point $\left(2\Delta^{-1},0\right)$ and is in the right half-plane of the complex plane. Furthermore, we can observe that $\phi^{-1}(\mathcal{C}) = \mathcal{D}$ is exactly the circle of radius $\lambda$ centered at $(0,0)$. Comparing \eqref{wnjUni} and \eqref{wnUni}, we can conclude. 
\end{proof}
%

%%%%%%%
\subsection{Divided differences formula and invertibility of gCQ based on the trapezoidal rule}

The definition of the weights in \eqref{weights} can be connected to Newton divided differences. 
This property was first noticed in \cite{LopezFernandezSauter2013} for BDF1 based gCQ, and relies on the following formula (see \cite[Equation (51)]{DeBoor2005}):  given a set of points $\{x_0,\ldots x_n\}$ and a complex analytic function $f$, then
\begin{equation} \label{DeBoor}
	\frac{1}{2 \pi \mi} \oint_{\mathcal{C}} f(s) \prod_{k=0}^n \left(s-x_k\right)^{-1} \dd s = [x_0, \ldots, x_n] f
\end{equation}
where $\mathcal{C}$ is a complex contour including the poles $\{x_0,\ldots,x_n\}$. Here, the divided difference $[x_m,\ldots,x_j] f$, for $0 \le m \le j \le n$, is defined in the classical way, iteratively by
\begin{equation*}
	[x_m,\ldots,x_j] f := 
	\begin{cases}
		\frac{[x_m,\ldots,x_{j-1}]f - [x_{m+1},\ldots,x_j]f}{x_m - x_j} & m < j, \\
		f(x_j) & m = j.
	\end{cases}
\end{equation*}
 To apply a formula similar to \eqref{DeBoor} in our situation we define an \textit{modified} divided difference formula
\begin{equation} \label{enDiff}
	\left\langle x_m,\ldots,x_j\right\rangle f := 
	\begin{cases} 
		(x_m+x_{m+1}) \left[x_m , \ldots , x_j \right] \left(f \prod_{k={m+2}}^j \left( x_k + \cdot \right)\right) & m < j, \\
		[x_j] f & m=j.
	\end{cases}
\end{equation}
We use formula \eqref{DeBoor} to see that the gCQ based on the trapezoidal rule can be written in a different form. First, we use definition \eqref{djn} of $D_j^n$ to see that
\begin{align*}
	\K\left(\partial_t^{\{\Delta_j\}}\right)g(t_n) & = \sum_{j=1}^{n-1} g^{(\rho)}(t_j) (-1)^{n-j+1} \left(2\Delta_j^{-1} + 2\Delta_{j+1}^{-1}\right) \frac{1}{2 \pi \mi} \oint_{\mathcal{C}} \K_\rho(s) \prod_{k=j+2}^n \left(s+2\Delta_k^{-1} \right) \prod_{k=j}^n \left(s-2\Delta_k^{-1}\right)^{-1} \dd s
	\\ & \hspace{0.5cm} - g^{(\rho)}(t_n)\frac{1}{2 \pi \mi} \oint_{\mathcal{C}} \K_\rho(s) \left(s-2\Delta_n^{-1}\right)^{-1} \dd s.
\end{align*}
Finally, by means of formula \eqref{DeBoor}, we obtain
\begin{align*}
	\K\left(\partial_t^{\{\Delta_j\}}\right)g(t_n) & = \sum_{j=1}^{n-1} g^{(\rho)}(t_j) (-1)^{n-j+1} \left(2\Delta_j^{-1} + 2\Delta_{j+1}^{-1}\right) \left[ 2\Delta_j^{-1}, \ldots, 2\Delta_n^{-1} \right] \left(\K_\rho \prod_{k=j+2}^n (2\Delta_k^{-1}+ \cdot) \right) 
	\\ & \hspace{0.3cm} - g^{(\rho)}(t_n) \left[2\Delta_n^{-1}\right] \K_\rho
	\\ & = \sum_{j=1}^n g^{(\rho)}(t_j) (-1)^{n-j+1} \langle 2\Delta_j^{-1},\ldots,2\Delta_n^{-1} \rangle \K_\rho.
\end{align*}
We introduce simplified notations for the subsequent results of this subsection. Specifically, we define for $0 \le m \le j \le n$, the polynomials $\PP_m^j \in \mathbb{P}^{j-m+1}(\C)$, and the operators $\D_m^j$, $\G_m^j$ as follows
\begin{equation} \label{DPG}
	\PP_m^j (z) := \prod_{k=m}^j (x_k+z), \quad \D_m^j(f) := [x_m, \ldots, x_j] f, \quad \G_m^j (f)  := \langle x_m, \ldots, x_j\rangle f.
\end{equation}
Our aim is to show that a Leibniz rule analogous to the standard one holds also for the modified divided difference \eqref{enDiff}. We state the following technical lemma.
\begin{lemma} \label{lemma1}
Given a set of points $\{x_0 , \ldots , x_n\} \subset \C$, then for all $n \ge 2$ and $2 \le \ell \le n-2$ it holds 
\begin{equation} \label{DivPP}
	\begin{cases}
		\sum_{k=\ell}^{n-2} \D_\ell^k(\PP_2^k) \D_k^k(\PP_{k+1}^{k+1}) \D_k^{n-1}(\PP_{k+2}^n) + \D_\ell^{n-1} (\PP_2^{n-1}) \D_{n-1}^{n-1} (\PP_n^n) = \D_\ell^{n-1} (\PP_2^n)  &\\
		\sum_{k=\ell}^j \D_\ell^k(\PP_2^k) \D_k^k(\PP_{k+1}^{k+1}) \D_k^j(\PP_{k+2}^n)  = \D_\ell^j (\PP_2^n) & j = \ell, \ldots, n-2
	\end{cases}
\end{equation}
where $\PP_h^j$ and $\D_h^j$ are defined in \eqref{DPG}.
\end{lemma}
\begin{proof}
These properties can be derived from the Leibniz product rule, which is applicable to a set of $N$ functions
\begin{equation}  \label{MULTLEB}
	\begin{aligned}
		\D_{\ell_1}^{\ell_2} \left(\varphi_1 \varphi_2 \cdots \varphi_N \right) & = \sum_{\ell_1 = \alpha_0 \le \alpha_1 \le \cdots \le \alpha_N = \ell_2} \D_{\ell_1}^{\alpha_1} (\varphi_1) \D_{\alpha_1}^{\alpha_2} (\varphi_2) \cdots \D_{\alpha_{N-1}}^{\ell_2} (\varphi_N)
		\\ & = \sum_{\ell_1  = \alpha_0 \le \alpha_1 \le \cdots \le \alpha_N = \ell_2} \prod_{\beta=0}^{N-1} \D_{\alpha_\beta}^{\alpha_{\beta+1}} (\varphi_{\beta+1}) 
	\end{aligned}
\end{equation}
the sum being over integers $\alpha_1, \ldots, \alpha_{N-1}$ such that $0 \le \alpha_1 \le \cdots \le \alpha_{N-1} \le n$. Specifically, we can use the multiplicative property $\PP_h^j = \prod_{q=h}^j \PP_q^q$ to partition both sides of \eqref{DivPP}, and then apply to each term \eqref{MULTLEB}.
\end{proof}

We are now able to state and prove the Leibniz rule for the modified divided difference \eqref{enDiff}. This will be the main tool to prove an inversion formula for the trapezoidal gCQ.
\begin{proposition}
Given a set of $n+1$ distinct points $\{x_0, \ldots, x_n\} \subset \C$ and two functions $f, g$ such that $f(x_i), g(x_i)$ are well-defined, then the following multiplicative rule holds
\begin{equation} \label{LRe}
	\langle x_0, \ldots , x_n\rangle(fg) = \sum_{k=0}^n \langle x_0, \ldots, x_k \rangle f \langle x_k,\ldots, x_n \rangle g.
\end{equation}
\end{proposition}
\begin{proof}
From definitions \eqref{enDiff} and \eqref{DPG} we observe that  
\begin{align*}
	\G_j^j(f) = \langle x_j\rangle f = [x_j] f = \D_j^j (f) \quad \text{~and~} \quad \G_{j-1}^j(f) = \langle x_{j-1},x_j\rangle f = (x_{j-1}+x_j) \D_{j-1}^j(f).
\end{align*}
Moreover, we observe that $\D_{j-1}^{j-1}(\PP_j^j) = [x_{j-1}] (x_j+\cdot) =  (x_{j-1}+x_j)$, from which we deduce
\begin{equation*}
	\G_{j-1}^j(f) = \D_{j-1}^{j-1}(\PP^j_j) \D_{j-1}^j (f).
\end{equation*}
For $0 \le m \le j-2 \le n$, we similarly obtain
\begin{align*}
	\G_m^j (f) = \langle x_m, \ldots, x_j\rangle f = (x_m+x_{m+1}) [x_m, \ldots, x_j] \left(f \prod_{k=m+2}^j (x_k + \cdot) \right) = \D_m^m \bigl(\PP_{m+1}^{m+1}\bigr)\D_m^j \bigl(f \PP_{m+2}^j\bigr).
\end{align*}
The standard Leibniz rule for divided differences can be written in the form
\begin{equation} \label{divdiff}
	\D_0^n (fg) = \sum_{\ell=0}^n \D_0^\ell (f) \D_\ell^n (g).
\end{equation}
Our goal is to demonstrate that the equation
\begin{equation} \label{claim}
	\G_0^n (fg) = \sum_{\ell=0}^n \G_0^\ell (f) \G_\ell^n (g)
\end{equation}
holds true for all non-negative integers $n$.

\tikzcircle{1.5pt} \textbf{Case $n=0$} \\
The statement \eqref{claim} is clear for $n=0$
\begin{equation*}
	\G_0^0 (fg) = \D_0^0 (fg) = f(x_0) g(x_0) = \D_0^0(f) \D_0^0 (g).
\end{equation*} 
\tikzcircle{1.5pt} \textbf{Case $n=1$} \\
We can readily determine the case where $n=1$ by applying the conventional Leibniz rule for divided differences, as expressed in equation \eqref{divdiff}
\begin{equation*}
	\G_0^1 (fg) = \D_0^0 (\PP_1^1) \D_0^1 (fg) = \D_0^0 (\PP_1^1) \left(\D_0^0 (f) \D_0^1 (g) + \D_0^1 (f) \D_1^1 (g) \right) = \G_0^0 (f) \G_0^1 (g) +  \G_0^1 (f) \G_1^1 (g).
\end{equation*}
\tikzcircle{1.5pt} \textbf{Case $n \ge 2$} \\
Now suppose $n \ge 2$. We split the proof in various subcases and sub-parts. Namely, we vary three indices $(n,j,\ell)$ in the following ranges $n \ge 0$, $0 \le \ell \le n$ and $\ell \le j \le n$. The rest of the proof has the following structure

\Tree[.{$n\ge 2$}
	[.{$\ell=0$} ]
	 [.{$\ell=1$} [.{$j=1$} ] [.{$2 \le j \le n-2$} ][.{$j=n-1$} ][.{$j = n$} ] ]
        [.{$2 \le \ell \le n$} [.{$\ell \le j \le n-2$} ][.{$j=n-1$} ][.{$j = n$} ] ]
        	[.{$\ell=n-1$} ]
	[.{$\ell=n$} ]   
 ]
 
By recalling the definitions \eqref{DPG} and the Leibniz rule for divided differences \eqref{divdiff}, for the left-hand side of \eqref{claim}, we can obtain
\begin{equation*} 
	\G_0^n (fg) = \D_0^0 (\PP_1^1)\D_0^n\left(f g \PP_2^n \right) = \D_0^0 (\PP_1^1) \sum_{\ell=0}^n \D_0^\ell (f) \D_\ell^n (g \PP_2^n) .
\end{equation*}
We split each term in the sum of the right hand side in \eqref{claim} by using again \eqref{DPG}
\begin{align*}
	\sum_{\ell=0}^n \G_0^\ell (f) \G_\ell^n (g) & = \D_0^0 (f) \D_0^0 (\PP_1^1) \D_0^n (g \PP_2^n) +  \D_0^0 (\PP_1^1) \D_0^1 (f) \D_1^1 (\PP_2^2) \D_1^n (g \PP_3^n)
	\\ &\hspace{0.5cm} + \sum_{\ell=2}^{n-2}  \D_0^0 (\PP_1^1) \D_0^\ell (f \PP_2^\ell) \D_\ell^\ell \bigl(\PP_{\ell+1}^{\ell+1}\bigr) \D_\ell^n (g \PP_{\ell+2}^n) 
	\\ & \hspace{0.5cm} + \D_0^0 (\PP_1^1) \D_0^{n-1} (f \PP_2^{n-1}) \D_{n-1}^{n-1} \left(\PP_n^n \right) \D_{n-1}^n \left(g\right)  + \D_0^0 (\PP_1^1) \D_0^n (f\PP_2^n) \D_n^n (g)
	\\ & = \D_0^0(\PP_1^1)  \sum_{\ell=0}^n \L_\ell^n (f,g),
\end{align*}
where we have defined the terms $\L_\ell^n(f,g)$ as
\begin{equation*}
	\begin{cases}
		\L_0^n(f,g) := \D_0^0(f) \D_0^n (g \PP_2^n), & \\
		\L_1^n(f,g) := \D_0^1 (f) \D_1^{1} \bigl(\PP_2^2\bigr) \D_1^n (g \PP_3^n), & \\
		\L_\ell^n(f,g) := \D_0^\ell (f \PP_2^\ell) \D_\ell^\ell \bigl(\PP_{\ell+1}^{\ell+1}\bigr) \D_\ell^n \bigl(g \PP_{\ell+2}^n\bigr), & \ell = 2,\ldots,n-2 \\
		\L_{n-1}^n(f,g) := \D_0^{n-1} (f \PP_2^{n-1})  \D_{n-1}^{n-1} \bigl(\PP_n^n\bigr) \D_{n-1}^n (g), & \\
		\L_n^n(f,g) := \D_0^n (f \PP_2^n)  \D_n^n (g). &
	\end{cases}
\end{equation*}

We rewrite each term $\L_\ell^n(f,g)$, for $2 \le \ell \le n$, using the Leibniz rule for $\D_0^\ell (f \PP_2^\ell)$
\begin{equation} \label{Ln}
	\begin{cases}
		\L_0^n(f,g) = \D_0^0(f) \D_0^n (g \PP_2^n), \\
		\L_1^n(f,g) = \D_0^1 (f) \D_1^{1} \bigl(\PP_2^2\bigr) \D_1^n (g \PP_3^n), \\
		\L_\ell^n(f,g) = \left(\sum_{h=0}^\ell \D_0^h (f) \D_h^\ell (\PP_2^\ell) \right) \D_\ell^\ell \bigl(\PP_{\ell+1}^{\ell+1}\bigr) \D_\ell^n \left(g \PP_{\ell+2}^n\right), \quad \ell = 2,\ldots,n-2 \\
		\L_{n-1}^n(f,g) = \left(\sum_{h=0}^{n-1} \D_0^h (f) \D_h^{n-1} \left(\PP_2^{n-1}\right) \right) \D_{n-1}^{n-1} \bigl(\PP_n^n\bigr) \D_{n-1}^n \left(g\right), \\
		\L_n^n(f,g) = \left(\sum_{h=0}^n \D_0^h (f) \D_h^n (\PP_2^n) \right) \D_n^n (g). \\
	\end{cases}
\end{equation}
We want to show that for each $\ell = 0, \ldots, n$, the sum of terms multiplying $\D_0^\ell (f)$ in \eqref{Ln} is exactly $\D_\ell^n(g \PP_2^n)$.

\tikzcircle{1.5pt} \textbf{Case $n \ge 2$, Subcase $\ell=0$} \\
When $\ell=0$, the only term in \eqref{Ln} that does not vanish is $\D_0^n (g \PP_2^n)$, which arises from $\L_0^n(f,g)$. For $\ell=2,\ldots,n$, the term $\D_0^0(f)$ multiplies $\D_0^\ell(\PP_2^\ell)$, which is zero since a divided difference with $\ell+1$ terms is always zero when applied to a polynomial of degree $\ell-1$.

\tikzcircle{1.5pt} \textbf{Case $n \ge 2$, Subcase $\ell=n$} \\
The only term for $\ell=n$ is
\begin{equation*}
	\D_n^n (g) \D_n^n (\PP_2^n) = \D_n^n (g \PP_2^n).
\end{equation*}

\tikzcircle{1.5pt} \textbf{Case $n \ge 2$, Subcase $\ell=n-1$} \\
The terms for $\ell=n-1$ are 
\begin{align*}
	\D_{n-1}^{n-1} \left(\PP_2^{n-1}\right) \D_{n-1}^{n-1} \left(\PP_n^n\right) \D_{n-1}^n (g) +  \D_{n-1}^n \left(\PP_2^n\right) \D_n^n (g) & =  \D_{n-1}^n \left(\PP_2^n\right) \D_n^n (g) + \D_{n-1}^{n-1} \left(\PP_2^{n}\right) \D_{n-1}^n (g)
	\\ & = \D_{n-1}^n (g \PP_2^n).
\end{align*}

\tikzcircle{1.5pt} \textbf{Case $n \ge 2$, Subcase $2 \le \ell \le n-2$} \\
We will now consider the case for general $\ell$, where $2 \leq \ell \leq n-2$. Within each $\L_k^n(f,g)$ for $k$ ranging from $\ell$ to $n$, there is a term that involves $\D_0^\ell(f)$. Hence, the sum that we are evaluating is expressed as follows:
\begin{equation}
\begin{aligned}
	& \sum_{k=\ell}^{n-2} \D_\ell^k (\PP_2^k) \D_k^k (\PP_{k+1}^{k+1}) \D_k^n (g \PP_{k+2}^n) + \D_\ell^{n-1} \left(\PP_2^{n-1}\right) \D_{n-1}^{n-1} (\PP_n^n) \D_{n-1}^n (g)
	\\ & \hspace{0.5cm} + \D_\ell^n \left(\PP_2^n\right) \D_n^n \left(g\right)
	\\ & = \sum_{k=\ell}^{n-2} \D_\ell^k (\PP_2^k) \D_k^k (\PP_{k+1}^{k+1}) \sum_{j=k}^{n} \D_k^j (\PP_{k+2}^n) \D_j^n (g)  + \D_\ell^{n-1} (\PP_2^{n-1}) \D_{n-1}^{n-1} (\PP_n^n) \D_{n-1}^n (g)
	\\ & \hspace{0.5cm} + \D_\ell^n (\PP_2^n ) \D_n^n \left(g\right) 
	\\ & = \sum_{j=\ell}^{n} \D_j^n (g) \sum_{k=\ell}^{\min\{j,n-2\}} \D_\ell^k (\PP_2^k) \D_k^k (\PP_{k+1}^{k+1}) \D_k^j(\PP^n_{k+2}) + \D_{n-1}^n (g) \D_\ell^{n-1} (\PP_2^{n-1}) \D_{n-1}^{n-1} (\PP_n^n) 
	\\ & \hspace{0.5cm} + \D_n^n (g) \D_\ell^n (\PP_2^n) \label{LHS2}
\end{aligned}
\end{equation}
where we used the standard Leibniz rule for each $\D_k^n(g \PP_{k+2}^n)$. We use again \eqref{divdiff} for $\D_\ell^n (g \PP_2^n)$
\begin{equation} \label{RHS2}
	\D_\ell^n (g \PP_2^n) = \sum_{j=\ell}^n \D_j^n (g) \D_\ell^j \left(\PP_2^n\right).
\end{equation}
Now we compare the terms near each $\D_j^n(g)$ for $j = \ell, \ldots n$ in \eqref{LHS2} and \eqref{RHS2}.

\tikzcircle{1.5pt} \textbf{Case $n \ge 2$, Subcase $2 \le \ell \le n-2$, Sub-part $j=n$} \\
The terms for $j=n$ are for both sides $\D_\ell^n (\PP_2^n)$, this because the first term in \eqref{LHS2} is zero. Indeed, $\D_k^j (\PP^n_{k+2}) = 0$ for all $k = \ell,\ldots,n-2$.

\tikzcircle{1.5pt} \textbf{Case $n \ge 2$, Subcase $2 \le \ell \le n-2$, Sub-part $j=n-1$} \\
The term for $j=n-1$ is $\D_\ell^{n-1} (\PP_2^n)$ in \eqref{RHS2}, while in \eqref{LHS2} it is
\begin{equation*}
	\sum_{k=\ell}^{n-2} \D_\ell^k(\PP_2^k) \D_k^k(\PP_{k+1}^{k+1}) \D_k^{n-1}(\PP_{k+2}^n) + \D_\ell^{n-1} (\PP_2^{n-1}) \D_{n-1}^{n-1} (\PP_n^n).
\end{equation*}
The two terms are the same thanks to Lemma \ref{lemma1}

\tikzcircle{1.5pt} \textbf{Case $n \ge 2$, Subcase $2 \le \ell \le n-2$, Sub-part $\ell \le j \le n-2$} \\
Similarly for $\ell \le j \le n-2$, in \eqref{RHS2} $\D_\ell^j (\PP_2^n)$, while in \eqref{LHS2} is
\begin{equation*}
	\sum_{k=\ell}^{j} \D_\ell^k(\PP_2^k) \D_k^k(\PP_{k+1}^{k+1}) \D_k^j(\PP_{k+2}^n).
\end{equation*}
The two terms are the same thanks to Lemma \ref{lemma1}.

\tikzcircle{1.5pt} \textbf{Case $n \ge 2$, Subcase $\ell=1$} \\
The term for $\ell=1$ is similar with the exception of a first new summand. Precisely, from $\L_k^n(f,g)$ with $k=1,\ldots,n$ we obtain the sum
\begin{align*} 
	& \D_1^{1} \left(\PP_2^2\right) \D_1^n (g \PP_3^n) + \sum_{k=2}^{n-2} \D_1^k (\PP_2^k) \D_k^k (\PP_{k+1}^{k+1}) \D_k^n (g \PP_{k+2}^n) + \D_1^{n-1} \left(\PP_2^{n-1}\right) \D_{n-1}^{n-1} (\PP_n^n) \D_{n-1}^n (g) + \D_1^n \left(\PP_2^n\right) \D_n^n \left(g\right).
\end{align*}
To conclude one can simply proceed as in the subcase $2 \le \ell \le n-2$, by considering also the new case $j=1$.
\end{proof}
Importantly, we can show that the inversion rule still holds for the gCQ based on the trapezoidal rule.
\begin{proposition}[Inversion formula]
Let $\{x_0,\ldots, x_n\} \subset \C$, and an operator $\K : \C_+ \to \B(X,Y)$ such that $\K^{-1}(x_i), i=0,\ldots,n$ are well defined, and $\{g_j\}_{j=0}^{n}, \{\phi_j\}_{j=0}^{n} \subset \mathbb{C}$. Then, the relation
\begin{align} \label{firstinv}
	\phi_n & = \sum_{j=0}^n (-1)^{n-j+1} g_j \left\langle x_j,\ldots,x_n\right\rangle \K
\end{align}
can be inverted and it holds that
\begin{align} \label{secondinv}
	g_n & = \sum_{\ell=0}^n (-1)^{n-\ell+1} \phi_\ell \left\langle x_\ell,\ldots,x_n\right\rangle \K^{-1}.
\end{align}
\end{proposition}
\begin{proof}
We proceed similarly as in \cite[Lemma 3.1]{LopezFernandezSauter2013}. We denote by $\widetilde{g}_n$ the left-hand side of \eqref{secondinv} when replacing $\phi_j$ by the definition \eqref{firstinv} in \eqref{secondinv}. Our aim is to show that $\widetilde{g}_n = g_n$. Using the Leibniz rule \eqref{LRe} for the modified divided difference we can write
\begin{align*}
	\widetilde g_n & = \sum_{j=0}^n (-1)^{n-j+1} \left\langle x_j,\ldots,x_n\right\rangle \K^{-1} \sum_{\ell=0}^j (-1)^{j-\ell+1} g_\ell \left\langle x_\ell,\ldots,x_j\right\rangle \K
	\\ & = \sum_{\ell=0}^n (-1)^{n-\ell} g_\ell \sum_{j=\ell}^n \left\langle x_j,\ldots,x_n\right\rangle \K^{-1} \left\langle x_\ell,\ldots,x_j\right\rangle \K
	\\ & =  \sum_{\ell=0}^n (-1)^{n-\ell} g_\ell \left\langle x_\ell,\ldots,x_n\right\rangle \mathds{1},
\end{align*}
where $\mathds{1}$ stands for the constant function $\mathds{1} \equiv 1$. It remains to show that $\langle x_\ell , \ldots, x_n \rangle \mathds{1} = \delta_\ell^n$. 

It is clear that $\langle x_n\rangle \mathds{1} = \left[x_n\right] \mathds{1} = 1$. To conclude we show that $\left\langle x_\ell,\ldots,x_n\right\rangle \mathds{1} = 0, $ for $\ell < n$, but this easily follows from definition \eqref{enDiff}
\begin{align*}
\left\langle x_\ell,\ldots,x_n\right\rangle \mathds{1} & = (x_\ell+x_{\ell+1}) \left[x_\ell,\ldots,x_n\right] \left( \prod_{\ell+2}^n(x_\ell+\cdot) \right) =  0,
\end{align*}
since a divided difference with $n-\ell+1$ terms applied to a polynomial of degree $n-\ell-1$ is zero.
\end{proof}

%%%%%%%%%%%%%%%%%%%%%
\section{Convergence analysis} \label{sec4}
Our convergence analysis follows the approach outlined in \cite[Section 2.7]{BanjaiSayas2022}. We demonstrate that utilizing the gCQ method is analogous to employing a specific composite midpoint technique coupled with an appropriate approximation of the integrand. It is important to note that our analysis assumes a smooth kernel that satisfies \eqref{assump} with $\mu < -3$. While this may be a limitation in theory, it is worth exploring whether this condition is truly necessary or simply an artificial construct for the purpose of our analysis. To this end, in Section \ref{sec5}, we conduct experiments to confirm whether this condition is indeed required. Further investigation in this area could yield valuable insights for practical applications.

We recall that the weights for an hyperbolic symbol $\mathcal{K} \in \mathcal{A}(\mu,\B(X,Y))$ of the gCQ based on the trapezoidal formula are defined by
\begin{equation*}
	w_{n,j}(\K) = D_j^n \frac{1}{2 \pi \mi} \oint_{\mathcal{C}} \K(s) G_j^n(s) \dd s
\end{equation*}
and the convolution can be rewritten as
\begin{equation*}
	\K\left(\partial_t^{\{ \Delta_j \}}\right) g(t_n) = \sum_{j=1}^n w_{n,j}(\K_\rho) g^{(\rho)}(t_j)
\end{equation*}
for $\rho \ge \lfloor \mu \rfloor+1$.
If $\mu < -1$, we can choose $\rho=0$, and the kernel $\kappa$ is simply the Laplace inverse of $\K$, precisely
\begin{equation} \label{inK}
	\kappa(t) = \frac{1}{2 \pi \mi} \int_{\sigma+ \mi \R} \K(s) e^{st} \dd s.
\end{equation}
We can now provide a precise estimate for symbols satisfying $\mu < -3$.
\begin{proposition} \label{prop4}
Let $\K \in \mathcal{A}(\mu, \mathcal{B}(X,Y))$ be a hyperbolic symbol with $\mu < -3$. Given $N+1$ time steps $0 = t_0 < t_1 < \ldots < t_N = T$, then
\begin{equation*}
	\left\| w_{n,j}(\K) - \frac{t_{j+1}-t_{j-1}}{2} \kappa\left(t_n - \frac{t_{j-1}+t_{j+1}}{2} \right) \right\|_{\B(X,Y)} \apprle (\Delta_j + \Delta_{j+1})\Delta_{\max}^2
\end{equation*}
for all $n = 1,\ldots,N$ and $j = 1, \ldots, n-1$, where the implicit constant may depend on $T, \mu$, but not on $\{ t_j\}$.
\end{proposition}
\begin{proof}
Let us fix $n$. We observe that, using \eqref{djn} and \eqref{weights}, for values of $j$ ranging from $1$ to $n-1$ it is possible to express
\begin{equation} \label{obs}
	D_j^n G_j^n(s) =  \frac{\Delta_j+\Delta_{j+1}}{2} \prod_{k=j+2}^n \left(1+ \frac{\Delta_k}{2}s\right) \prod_{k=j}^n \left(1- \frac{\Delta_k}{2}s\right)^{-1}.
\end{equation}
Recalling  the definition of the weights \eqref{weights} (with the integration now over a complex line $\sigma+ \mi \R$) combined with \eqref{obs}, and \eqref{inK}, we start by writing
\begin{align*}
	& \left\| w_{n,j}(\K) - \frac{t_{j+1}-t_{j-1}}{2} \kappa\left(t_n - \frac{t_{j-1}+t_{j+1}}{2} \right) \right\|_{\mathcal{B}(X,Y)}
	\\ & \hspace{0.5cm}  = \frac{\Delta_j + \Delta_{j+1}}{4\pi} \left\| \int_{\sigma + \mi \R} \K(s) \left(\prod_{k=j+2}^n \left(1+ \frac{\Delta_k}{2}s\right) \prod_{k=j}^n \left(1- \frac{\Delta_k}{2}s\right)^{-1} - e^{s\Bigl(t_n - \frac{t_{j-1}+t_{j+1}}{2}\Bigr)} \right) \dd s \right\|_{\mathcal{B}(X,Y)}
	\\ & \hspace{0.5cm}  = \frac{\Delta_j + \Delta_{j+1}}{4\pi} \left\| \int_{\sigma + \mi \R} \K(s) \mathcal{I}_j^n(s) \dd s \right\|_{\mathcal{B}(X,Y)}
\end{align*}
where   
\begin{equation} \label{Ijn}
\mathcal{I}_j^n(s) : =  \prod_{k=j+2}^n \left(1+ \frac{\Delta_k}{2}s\right) \prod_{k=j}^n \left(1- \frac{\Delta_k}{2}s\right)^{-1} - e^{s\Bigl(t_n - \frac{t_{j-1}+t_{j+1}}{2}\Bigr)}. 
\end{equation}
Notice that we also used that $t_{j+1} - t_{j-1} = \Delta_j + \Delta_{j+1}$. 

We perform a simple manipulation on the first tern
\begin{equation} \label{q1}
	\prod_{k=j+2}^n \left(1+ \frac{\Delta_k}{2}s\right) \prod_{k=j}^n \left(1- \frac{\Delta_k}{2}s\right)^{-1} =  \left(1- \frac{\Delta_j}{2}s\right)^{-1} \left(1- \frac{\Delta_{j+1}}{2}s\right)^{-1} \prod_{k=j+2}^n \frac{\left(1+ \frac{\Delta_k}{2}s\right)}{\left(1- \frac{\Delta_k}{2}s\right)}.
\end{equation}

Then, we observe that $t_n - \frac{t_{j-1}+t_{j+1}}{2} = \sum_{k=j+2}^n \Delta_k + \frac{\Delta_j + \Delta_{j+1}}{2}$, from which we deduce
\begin{equation} \label{q2}
	e^{s\left(t_n - \frac{t_{j-1}+t_{j+1}}{2}\right)}  = e^{s\frac{\Delta_j}{2}} e^{s\frac{\Delta_{j+1}}{2}}\prod_{k=j+2}^n e^{s \Delta_k}.
\end{equation}
Our current objective is to compute the quantities \eqref{q1} and \eqref{q2} for small arguments and then compare them. 

If $\abs{s \Delta_{\max}}$ is sufficiently small, we can expand \eqref{q2} in Taylor series and readily obtain
\begin{equation} \label{q3}
\begin{aligned} 
	e^{s\frac{\Delta_j}{2}} e^{s\frac{\Delta_{j+1}}{2}}\prod_{k=j+2}^n e^{s \Delta_k} =  1 & + s \left( \frac{\Delta_j + \Delta_{j+1}}{2} + \sum_{k=j+2}^n \Delta_k \right)
	\\ & + \frac{1}{2} s^2 \left(\frac{\Delta_{j}+\Delta_{j+1}}{2}+ \sum_{k=j+2}^n \Delta_k \right)^2 + \mathcal{O}(s^3\Delta_{\max}^3). 
\end{aligned}
\end{equation}
Similarly, for \eqref{q1} we deduce when $\abs{s \Delta_{\max}} \to 0$ 
\begin{equation} \label{q4}
\begin{aligned} 
	& \left(1- \frac{\Delta_j}{2}s\right)^{-1} \left(1- \frac{\Delta_{j+1}}{2}s\right)^{-1} \prod_{k=j+2}^n \frac{\left(1+ \frac{\Delta_k}{2}s\right)}{\left(1- \frac{\Delta_k}{2}s\right)} 
	\\ & \hspace{0.5cm} = \left(1+\frac{\Delta_j}{2}s+\frac{\Delta_j^2}{4} s^2\right)\left(1+\frac{\Delta_{j+1}}{2}s+\frac{\Delta_{j+1}^2}{4}s^2\right) \prod_{k=j+2}^n \left(1+\Delta_ks + \frac{\Delta_k^2}{2} s^2 \right) + \mathcal{O}(s^3 \Delta_{\max}^3) 
	\\ & \hspace{0.5cm} = 1  + s \left( \frac{\Delta_j + \Delta_{j+1}}{2} + \sum_{k=j+2}^n \Delta_k \right) 
	\\ & \hspace{1.15cm} +\frac{1}{2}s^2 \left(\frac{\Delta_j^2+\Delta_{j+1}^2}{2} + \sum_{k=j+2}^n \Delta_k^2+ \frac{\Delta_j \Delta_{j+1}}{2} + (\Delta_j + \Delta_{j+1}) \sum_{k=j+2}^n \Delta_k + 2\sum_{\underset{k_1,k_2=j+2}{k_1 \ne k_2}}^n \Delta_{k_1} \Delta_{k_2} \right) 
	\\ & \hspace{1.15cm} + \mathcal{O}(s^3 \Delta_{\max}^3). 
\end{aligned}
\end{equation}
Combining \eqref{q1}, \eqref{q2}, \eqref{q3} and \eqref{q4} we deduce, for $\abs{s \Delta_{\max}}$ small enough,
\begin{align*}
	\mathcal{I}_j^n(s)  = - \frac{1}{2}s^2 \left(\frac{\Delta_j^2 + \Delta_{j+1}^2}{4} \right) + \mathcal{O}(s^3\Delta_{\max}^3).
\end{align*}
Consequently, we can split the integral for $\abs{s \Delta_{\max}} < c$ with $c$ small enough and $\abs{s\Delta_{\max}} > c$ 
\begin{equation} \label{xlq1}
\begin{aligned}
	& \left\| w_{n,j}(\K) - \frac{t_{j+1}-t_{j-1}}{2} \kappa\left(t_n - \frac{t_{j-1}+t_{j+1}}{2} \right) \right\|_{\B(X,Y)} 
	\\ & \hspace{0.5cm}= \frac{\Delta_j + \Delta_{j+1}}{4\pi} \left\| \int_{\sigma + \mi \R} \K(s) \mathcal{I}_j^n(s) \dd s \right\|_{\B(X,Y)}
	\\ & \hspace{0.5cm} \apprle (\Delta_j + \Delta_{j+1}) \left( \left\| \int_{\sigma + \mi \mathbb{R}, \abs{s \Delta_{\max}} < c} \K(s) \mathcal{I}_j^n(s) \dd s \right\|_{\B(X,Y)} + \left\| \int_{\sigma + \mi \R, \abs{s \Delta_{\max}} > c} \K(s)\mathcal{I}_j^n(s) \dd s \right\|_{\B(X,Y)} \right)
	\\ & \hspace{0.5cm} \apprle (\Delta_j + \Delta_{j+1})\left((\Delta_j^2 + \Delta_{j+1}^2) \int_{\sigma + \mi \R, \abs{s \Delta_{\max}} < c}  |s|^{\mu+2}\dd s
 + \int_{\sigma + \mi \R, \abs{s \Delta_{\max}} > c} |s|^{\mu} \left|  \mathcal{I}_j^n(s) \right| \dd s\right)
\end{aligned}
\end{equation}
where in the last we used \eqref{assump}. We deduce now two auxiliary results to bound $| \mathcal{I}_j^n(s)|$ for $\abs{s \Delta_{\max}} > c$.  

For $\Re s \in [0, \nicefrac{1}{2})$, we readily verify that
\begin{equation*}
	\left|\frac{1+s}{1-s} \right| \le \frac{1+\Re s}{1- \Re s} = 1 + \frac{2 \Re s}{1 - \Re s} \le e^{\frac{2 \Re s}{1- \Re s}} \le e^{4 \Re s}
\end{equation*}
and
\begin{equation*}
	\left|1-s \right|^{-1} \le \frac{1}{1- \Re s} = 1 + \frac{\Re s}{1 - \Re s} \le e^{\frac{ \Re s}{1- \Re s}} \le e^{2 \Re s}.
\end{equation*}
From the latter, recalling definition \eqref{Ijn}, we obtain
\begin{equation} \label{xlq2}
\begin{aligned}
	\abs{\mathcal{I}_j^n(s)} & = \left| 1-\frac{\Delta_j}{2} s\right|^{-1} \left| 1-\frac{\Delta_{j+1}}{2} s\right|^{-1} \prod_{k=j+2}^n \left| \frac{1+ \frac{\Delta_k}{2}s}{1- \frac{\Delta_k}{2}s} \right|
	\\ & \le e^{(\Delta_j+\Delta_{j+1}) \Re s} e^{2 (\sum_{k=j+2}^n \Delta_k) \Re s} 
	\\ &\le e^{2(t_n - t_j) \Re s} \le e^{2T\Re s}.
\end{aligned}
\end{equation}
Finally, combining \eqref{xlq1} and \eqref{xlq2},  we obtain
\begin{align*}
	\left\| w_{n,j}(\K) - \frac{t_{j+1}-t_{j-1}}{2} \kappa\left(t_n - \frac{t_{j-1}+t_{j+1}}{2} \right) \right\|_{\mathcal{B}(X,Y)} & \apprle (\Delta_j+\Delta_{j+1}) \Delta_{\max}^2 \int_{\sigma + \mi \R} |s|^{\mu+2}\dd s
\end{align*}
which is bounded for $\mu < -3$.
\end{proof}

\begin{remark} [Why do we obtain an extra order of convergence compared with BDF1?]
In the BDF1 setting, see \cite{LopezFernandezSauter2013} for details, the convolution $K(\partial_t)g$ is approximated at the time step $t_n$ by
\begin{align*}
	\K\left({\partial_t}^{\{ \Delta_j,{\text{BDF1}} \}}\right) g (t_n) := \sum_{j=1}^n g(t_j) \Delta_j \frac{1}{2 \pi \mi} \oint_{\mathcal{C}} \K(s) B_j^n(s) \dd s =: \sum_{j=1}^n g(t_j) \omega_{n,j}^{\text{BDF1}}(\mathcal{K})
\end{align*}
where the complex contour $\mathcal{C}$ includes the complex poles $\Delta_k^{-1}$ and  $B_j^n(s) := \prod_{k=j}^n \left(1- \Delta_k s\right)^{-1}$.

For $\abs{s \Delta_{\max}}$ small enough, we can proceed as in the proof of Proposition \ref{prop4} and readily obtain 
\begin{equation} \label{NNN22}
	e^{s(t_n - t_{j-1})} - B_j^n(s) = s^2 \sum_{\underset{k_1 \ne k_2}{k_1,k_2=j}}^n \Delta_{k_1} \Delta_{k_2}  + \mathcal{O}(s^3 \Delta_{\max}^3).
\end{equation}
From which we can deduce (see details in \cite[Proposition 2.31]{BanjaiSayas2022})
\begin{equation*}
	\left\| w^{\text{BDF1}}_{n,j}(\K) - \Delta_j \kappa\left(t_n - t _{j-1}\right) \right\|_{\B(X,Y)} \apprle \Delta_j \Delta_{\max}
\end{equation*}
for hyperbolic symbols satisfying $\K \in \mathcal{A}(\mu,\B(X,Y))$ with $\mu > -3$.

In \eqref{NNN22}, we can only place an upper limit on the coefficient of $s^2$, which is $\Delta_{\max}^2 N \apprle \Delta_{\max}$. This is in contrast to the bound achieved by the trapezoidal rule, where we obtain one extra order of convergence. However, we still require the same level of smoothness for $\K$.
\end{remark}
In order to establish our main result, it is essential to develop a novel quadrature formula that can accurately evaluate integrals with integrands that vanish at the limits of the integration interval. This formula plays a pivotal role in our analysis since we will utilize it to compare the gCQ discretization with a similar quadrature formula in our main proof.
\begin{lemma} \label{lemma2}
Let $t_0<t_1< \ldots< t_n$ and set $\Delta_j = t_j - t_{j-1}$. Given $f \in C^2([t_0,t_n],Y)$ such that 
\begin{equation*}
	f(t_0)=f(t_n)=0,
\end{equation*}
we define the integration rule
\begin{equation*}
	\mathcal{Q}^{\{\Delta_j\}}(f) := \sum_{j=1}^{n-1} \left( \frac{t_{j+1} - t_{j-1}}{2} \right) f \left( \frac{t_{j+1} + t_{j-1}}{2} \right).
\end{equation*}
Then, the following holds
\begin{equation} \label{prop5}
\begin{aligned}
	 \left\| \int_{t_0}^{t_n} f(t) \dd t - \mathcal{Q}^{\{\Delta_j\}}(f) \right\|_Y & \apprle \Delta_1^3 \max_{t \in [t_0,t_1]}\|f^{(2)}(t)\|_Y  + \sum_{j=1}^{n-1} \left( \Delta_j + \Delta_{j+1} \right)^3 \max_{t \in [t_{j-1},t_{j+1}]} \bigl\|f^{(2)}(t)\bigr\|_Y
	 \\ & \hspace{0.3cm} + \Delta_n^3 \max_{t \in [t_{n-1},t_n]} \|f^{(2)}(t)\|_Y
	 + \Delta_1^2\| f^{(1)}(t_0)\|_Y + \Delta_n^2\|f^{(1)}(t_n)\|_Y. 
\end{aligned}
\end{equation}
\end{lemma}
\begin{proof}
To prove inequality \eqref{prop5}, we utilize the key idea of viewing the new quadrature rule $\mathcal{Q}^{\{\Delta_j\}}$ as a combination of two composite midpoint rules. Specifically, one associated to the grid $\{t_0,t_2,\ldots,t_n\}$ and a second one to $\{t_1,t_3,\ldots,t_{n-1}\}$. We recall that the local midpoint rule is a numerical method for approximating integrals, where the integrand is evaluated at the midpoint of the integration interval. The local quadrature error of this method is given by:
\begin{equation} \label{local_error}
\left\| \int_a^b f(t) \dd t - (b-a) f\left( \frac{a+b}{2} \right) \right\|_Y \le \frac{(b-a)^3}{24} \max_{t \in [a,b]} \| f^{(2)}(t) \|_Y,
\end{equation}
where $f : [a,b] \to Y$ is the integrand function, and $[a,b] \subset \mathbb{R}$ the integration interval.
Let $n$ be even; if $n$ is odd the proof is similar. Using \eqref{local_error} with the grid $\{t_0,t_2,\ldots,t_n\}$, we observe that
\begin{align} \label{PP1}
	\left\| \int_{t_0}^{t_n} f(t) \dd t \right. & \left. - \sum_{k=0}^{\frac{n}{2}-1} \left( t_{2k+2} - t_{2k} \right) f \left( \frac{t_{2k+2} + t_{2k}}{2} \right) \right\|_Y \le \sum_{k=0}^{\frac{n}{2}-1} \frac{\left(t_{2k+2} - t_{2k} \right)^3}{24} \max_{t \in [t_{2k},t_{2k+2}]} \|f^{(2)}(t)\|_Y,
\end{align}
and similarly
\begin{equation} \label{PP2}
	\begin{aligned} 
		& \left\| \int_{t_1}^{t_{n-1}} f(t) \dd t - \sum_{k=1}^{\frac{n}{2}-2} \left(t_{2k+1} - t_{2k-1} \right) f \left( \frac{t_{2k+1} + t_{2k-1}}{2} \right)\right\|_Y
		\\ & \hspace{7.5cm} \le \sum_{k=1}^{\frac{n}{2}-2} \frac{\left(t_{2k+1} - t_{2k-1} \right)^3}{24} \max_{t \in [t_{2k-1},t_{2k+1}]} \|f^{(2)}(t)\|_Y.
	\end{aligned}
\end{equation}
We can conclude using in the remaining intervals $[t_0,t_1]$ and $[t_{n-1},t_n]$ the trapezoidal rule (with the local error similar to \eqref{local_error} but with the constant $\nicefrac{1}{12}$)
\begin{equation} \label{PP3}
	\begin{aligned}
	\left\| 2 \int_{t_0}^{t_n} f(t) \dd t \right\|_Y & \le \left\| \int_{t_0}^{t_n} f(t) \dd t \right\|_Y + \left\| \int_{t_1}^{t_{n-1}} f(t) \dd t \right\|_Y + \left\| \int_{t_0}^{t_1} f(t) \dd t \right\|_Y + \left\| \int_{t_{n-1}}^{t_n} f(t) \dd t \right\|_Y
	\\ & \le \left\|\int_{t_0}^{t_n} f(t) \dd t \right\|_Y + \left\|\int_{t_1}^{t_{n-1}} f(t) \dd t \right\|_Y + \frac{\Delta_1}{2} \left\|f(t_1)\right\|_Y + \frac{\Delta_1^3}{12} \max_{t \in [t_0,t_1]} \| f^{(2)}(t) \|_Y
	\\ & \hspace{5.8cm} + \frac{\Delta_n}{2} \|f(t_{n-1})\|_Y + \frac{\Delta_n^3}{12} \max_{t \in [t_{n-1},t_n]} \| f^{(2)}(t) \|_Y
	\\ & \apprle \left\|\int_{t_0}^{t_n} f(t) \dd t \right\|_Y + \left\|\int_{t_1}^{t_{n-1}} f(t) \dd t \right\|_Y + \Delta_1^2 \left\|f^{(1)}(t_0)\right\|_Y + \Delta_1^3 \max_{t \in [t_0,t_1]} \| f^{(2)}(t) \|_Y
	\\ & \hspace{5.75cm} + \Delta_n^2 \|f^{(1)}(t_{n})\|_Y + \Delta_n^3 \max_{t \in [t_{n-1},t_n]} \| f^{(2)}(t) \|_Y
	\end{aligned}
\end{equation}
since $f(t_1) =  f'(t_0) \Delta_1 + \frac{f''(\xi_1) \Delta_1^2}{2}$ and $f(t_{n-1}) = - f'(t_n) \Delta_n + \frac{f''(\xi_n) \Delta_n^2}{2}$ for $\xi_1 \in [t_0,t_1]$ and $\xi_n \in [t_{n-1},t_n]$.

We conclude combining \eqref{PP1}, \eqref{PP2} and \eqref{PP3}, recalling that for $n$ even
\begin{equation*}
	2\mathcal{Q}^{\{ \Delta_j\}}(f) = \sum_{k=0}^{\frac{n}{2}-1} \left( t_{2k+2} - t_{2k} \right) f \left( \frac{t_{2k+2} + t_{2k}}{2} \right) + \sum_{k=1}^{\frac{n}{2}-2} \left(t_{2k+1} - t_{2k-1} \right) f \left( \frac{t_{2k+1} + t_{2k-1}}{2} \right).
\end{equation*}
\end{proof}

\begin{theorem}
Let $\K \in \mathcal{A}(\mu,\mathcal{B}(X,Y))$, $\mu \in \R$ and $\rho > \max\{-1,\mu+3\}$. Consider a casual function $g \in C^{\rho-1}(\R)$ satisfying $g^{(j)}(0) =0 $, $j = 0,\dots,\rho-1$, and $g^{(\rho)}$ locally integrable. Then, it holds
\begin{align*}
	 & \left\| \K(\partial_t)g(t_n) - \K_\rho\left(\partial_t^{\{\Delta_j\}}\right)g^{(\rho)}(t_n) \right\|_Y 
	 \\ & \hspace{3cm} \apprle \Delta_1^3\max_{t \in [0,t_1]}\|g^{(\rho+2)}(t)\|_X + \sum_{j=1}^{n-1} (\Delta_j + \Delta_{j+1})^3  \max_{t \in [t_{j-1},t_{j+1}]} \|g^{(\rho+2)}(t)\|_X
	 \\ & \hspace{3.3cm} + \Delta_n^3\max_{t \in [t_{n-1},t_n]}\|g^{(\rho+2)}(t)\|_X  + \Delta_{\max}^2 \sum_{j=1}^{n-1} (\Delta_j + \Delta_{j+1}) \| g^{(\rho)}(t_j) \|_X	
	 \\ & \hspace{3.3cm} + \Delta_1^2 \| g^{(\rho+1)}(0) \|_X + \sum_{j=1}^{n-1} (\Delta_j + \Delta_{j+1})|\Delta_{j+1}-\Delta_j| \max_{t \in [t_{j-1},t_{j+1}]} \|g^{(\rho+1)}(t)\|_X.
\end{align*}
In particular, if $| \Delta_{j+1} - \Delta_j | \apprle \Delta_{\max}^2$, then we have 
\begin{align*}
	& \left\| \K(\partial_t)g(t_n) -\K_\rho\left(\partial_t^{\{\Delta_j\}}\right)g^{(\rho)}(t_n) \right\|_Y
	\\ & \hspace{0.2cm} \apprle \Delta_{\max}^2 \Biggl[\sum_{j=1}^{n-1} (\Delta_j + \Delta_{j+1}) \left( \max_{t \in [t_{j-1},t_{j+1}]} \|g^{(\rho+2)}(t)\|_X + \max_{t \in [t_{j-1},t_{j+1}]} \|g^{(\rho+1)}(t)\|_X + \| g^{(\rho)} (t_j)\|_X \right) \Biggr.
	\\ & \hspace{1.7cm} \Biggl. + \Delta_1 \max_{t \in [0,t_1]}\|g^{(\rho+2)}(t)\|_X + \Delta_n \max_{t \in [t_{n-1},t_n]}\|g^{(\rho+2)}(t)\|_X + \| g^{(\rho+1)}(0) \|_X \Biggr].
\end{align*}
\end{theorem}
\begin{proof}
By definition \eqref{convolution}, we can write
\begin{equation} \label{convEQ}
	\K(\partial_t) g(t_n) = \int_0^{t_n} \kappa_\rho(t_n - \tau) g^{(\rho)}(\tau) \dd \tau.
\end{equation}
Let us define $\widetilde{I}_n$ to be a first approximation of this integral
\begin{equation*}
	\widetilde{I}_n := \sum_{j=1}^{n-1} \left( \frac{t_{j+1} - t_{j-1}}{2} \right) \kappa_\rho \left(t_n - \frac{t_{j+1} + t_{j-1}}{2} \right) g^{(\rho)}(t_j).
\end{equation*}
We express the error in two terms $\K(\partial_t)g(t_n)- \K_\rho\left(\partial_t^{\{\Delta_j\}}\right)g^{(\rho)}(t_n) = E_1 + E_2$ where
\begin{equation*}
	E_1 := \K(\partial_t) g(t_n) - \widetilde{I}_n, \qquad	E_2 :=  \widetilde{I}_n - \K_\rho\left(\partial_t^{\{\Delta_j\}}\right)g^{(\rho)}(t_n).
\end{equation*}
To bound $E_1$ we notice that 
\begin{equation} \label{SS1}
\begin{aligned}
	\left\| E_1\right\|_Y & =\left\| \K(\partial_t) g(t_n) - \sum_{j=1}^{n-1} \left( \frac{t_{j+1} - t_{j-1}}{2} \right) \kappa_\rho \left(t_n - \frac{t_{j+1} + t_{j-1}}{2} \right) g^{(\rho)}(t_j) \right\|_Y
	 \\ & \apprle \left\|\K(\partial_t) g(t_n) - \sum_{j=1}^{n-1} \left( \frac{t_{j+1} - t_{j-1}}{2} \right) \kappa_\rho \left(t_n - \frac{t_{j+1} + t_{j-1}}{2} \right) g^{(\rho)}\left(\frac{t_{j+1}+t_{j-1}}{2}\right)\right\|_Y 
	 \\ & \hspace{0.5cm}  + \sum_{j=1}^{n-1} \left( \frac{t_{j+1} - t_{j-1}}{2} \right) \left\| \kappa_\rho \left(t_n - \frac{t_{j+1} + t_{j-1}}{2} \right) \right\|_{\B(X,Y)} \left\| g^{(\rho)}(t_j) - g^{(\rho)}\left(\frac{t_{j+1}+t_{j-1}}{2}\right)\right\|_X 
	 \\ & =: I_1 + I_2.
\end{aligned}
\end{equation}
Recalling \eqref{convEQ} and applying Lemma \ref{lemma2}, we can estimate
\begin{equation} \label{SS2}
\begin{aligned}
	I_1 & = \left\| \int_{0}^{t_n} \kappa_\rho(t_n - \tau) g^{(\rho)}(\tau) \dd \tau - \mathcal{Q}^{\{\Delta_j\}}\left(\kappa_\rho(t_n - \cdot) g^{(\rho)}\right) \right\|_Y
	\\ & \apprle \Delta_1^3\max_{t \in [0,t_1]}\|g^{\rho+2}(t)\|_X + \sum_{j=1}^{n-1} (\Delta_j + \Delta_{j+1})^3 \max_{t \in [t_{j-1},t_{j+1}]} \|g^{(\rho+2)}(t)\|_X  + \Delta_n^3\max_{t \in [t_{n-1},t_n]}\|g^{(\rho+2)}(t)\|_X
	\\ & \hspace{0.5cm} + \Delta_1^2 \| g^{(\rho+1)}(0) \|_X,
\end{aligned}
\end{equation}
where we also used the boundedness of $\| \kappa_\rho(\cdot) \|_{\B(X,Y)}$ and $\kappa_\rho(0) = \partial_t \kappa_\rho(0) = 0$.

On the other hand, we observe that 
\begin{equation*}
	 \left\| g^{(\rho)}(t_j) - g^{(\rho)}\left(\frac{t_{j+1}+t_{j-1}}{2}\right)\right\|_X \apprle \left| \frac{t_{j+1} + t_{j-1} - 2t_j}{2} \right| \max_{t \in [t_{j-1},t_{j+1}]} \| g^{(\rho+1)}(t) \|_X,
\end{equation*}
from which we conclude
\begin{equation} \label{SS3}
\begin{aligned}
	I_2 & \apprle \sum_{j=1}^{n-1} \left( \frac{t_{j+1}-t_{j-1}}{2}\right) \left| \frac{t_{j+1} + t_{j-1} - 2t_j}{2} \right| \max_{t \in [t_{j-1},t_{j+1}]} \| g^{(\rho+1)}(t) \|_X
	\\ & \apprle \sum_{j=1}^{n-1} (\Delta_j + \Delta_{j+1}) | \Delta_{j+1} - \Delta_j | \max_{t \in [t_{j-1},t_{j+1}]} \| g^{(\rho+1)}(t) \|_X.
\end{aligned}
\end{equation}
To bound $E_2$ we simply use Proposition \ref{prop4}
\begin{equation} \label{SS4}
\begin{aligned}
	\left| E_2 \right| & \le \sum_{j=1}^{n-1} \left\| w_{n,j}(\K_\rho) - \frac{t_{j+1} - t_{j-1}}{2} \kappa_\rho\left( t_n - \frac{t_{j-1} + t_{j+1}}{2}\right) \right\|_X \| g^{(\rho)}(t_j) \|_{\B(X,Y)}
	\\ & \apprle  \Delta_{\max}^2 \sum_{j=1}^{n-1} (\Delta_j + \Delta_{j+1}) \| g^{(\rho)}(t_j)\|_X. 
\end{aligned}
\end{equation}
Combining estimates \eqref{SS1}, \eqref{SS2}, \eqref{SS3} and \eqref{SS4} we conclude.
\end{proof}

\begin{remark}
In practice, a common choice of variable grids is a graded mesh of the form $\{t_j = \left( \nicefrac{j}{N} \right)^{\alpha} j = 0,\ldots N\}$, for some $\alpha \ge 1$ and $N \in \mathbb{N}$. For this particular time stepping schemes, we can verify that
\begin{equation} \label{condmax}
	| \Delta_{j+1} - \Delta_j| \le \Delta_{\max}^2
\end{equation}
for all $j = 1,\ldots,N-1$. Indeed, we reach the maximum for $j=N-1$, thus obtaining
\begin{equation} \label{qqq1}
\begin{aligned}
	\max_{j \in \{1,\ldots,N-1\}}	| \Delta_{j+1} - \Delta_j| & = \left(1-\frac{(N-1)^\alpha}{N^\alpha}\right) - \left(\frac{(N-1)^\alpha}{N^\alpha}-\frac{(N-2)^\alpha}{N^\alpha}\right) 
	\\ & = 1 - 2\frac{(N-1)^\alpha}{N^\alpha} + \frac{(N-2)^\alpha}{N^\alpha}
\end{aligned}
\end{equation}
and similarly, since $\Delta_{\max} = \Delta_N$,
\begin{equation} \label{qqq2}
	\Delta_{\max}^2 =  \left(1 -\frac{(N-1)^\alpha}{N^\alpha}\right)^2 = 1 - 2\frac{(N-1)^\alpha}{N^\alpha} + \frac{(N-1)^{2\alpha}}{N^{2\alpha}}.
\end{equation}
Combining \eqref{qqq1} and \eqref{qqq2}, we conclude that \eqref{condmax} is equivalent to
\begin{equation*}
	  \frac{(N-2)^\alpha}{N^\alpha} \le \frac{(N-1)^{2\alpha}}{N^{2\alpha}},
\end{equation*}
but this is clearly true for all $N \in \mathbb{N}$ and for all $\alpha \ge 1$.
\end{remark}
%%
%%%%%
\section{Numerical results and algorithms} \label{sec5}
This section outlines the numerical algorithms used to compute the forward and backward gCQ based on the trapezoidal rule. Additionally, we introduce the gCQ based on BDF2 with non-uniform time steps, along with the corresponding algorithms. We also provide a reminder of the quadrature rules proposed in \cite{LopezFernandezSauter2015b} and explain how we have adapted them to our specific context. To illustrate the effectiveness of our proposed methods, we include a one-dimensional numerical example.
\subsection{gCQ based on BDF2 with non-uniform steps}
Here, we present also a gCQ method based on BDF2 on variable grids. Proceeding as in \eqref{convo} we need to discretize the initial value problem \eqref{ode} via a non uniform BDF2 scheme (see e.g. \cite[Section 5]{Grigorieff1983}). 

Given $0 = t_0 < t_1 < \ldots < t_N = T$ with time-steps $\Delta_n = t_n - t_{n-1}, n = 1, \ldots, N$, and setting $\Delta_0 = \Delta_1$, the approximation at $t_n$ of the solution of \eqref{ode}, for $n = 1,\ldots,N$ is
\begin{align*}
	u_n(s) =  u_{n-1}(s) \frac{(\Delta_{n-1} + \Delta_n)^2}{\Delta_{n-1}(\Delta_{n-1}+2\Delta_n)} - u_{n-2}(s) \frac{\Delta_{n}^2}{\Delta_{n-1}(\Delta_{n-1}+2\Delta_{n})} + \bigl( su_n(s) + g^{(\rho)}(t_n) \bigr) \frac{\Delta_{n}(\Delta_{n-1}+\Delta_{n})}{\Delta_{n-1}+2\Delta_{n}}
\end{align*}
with $u_0(s) \equiv u_{-1}(s) \equiv 0$, from which we simplify
\begin{align} \label{solBDF2}
	u_n(s) = u_{n-1}(s) \frac{B_n}{1-A_n s} - u_{n-2}(s) \frac{C_n}{1-A_ns} + g^{(\rho)}(t_n) \frac{A_n}{1-A_n s} 
\end{align}
with the defined coefficients
\begin{equation} \label{coefficients}
	A_n := \frac{\Delta_{n}(\Delta_{n-1}+\Delta_{n})}{\Delta_{n-1}+2\Delta_{n}}, \quad B_n := \frac{(\Delta_{n-1} + \Delta_n)^2}{\Delta_{n-1}(\Delta_{n-1}+2\Delta_n)}, \quad C_n := \frac{\Delta_{n}^2}{\Delta_{n-1}(\Delta_{n-1}+2\Delta_{n})}.
\end{equation}
Deriving a closed-form solution similar to equation \eqref{recursion} for this recursive approach is a challenging task. Hence, we have decided to postpone the theoretical analysis of BDF2 gCQ for future research. In this paper, we focus on providing a concise description of Algorithms \ref{alg2} and \ref{alg4} on page \pageref{alg2} that facilitate the computation of forward and backward gCQ utilizing BDF2. Additionally, we conduct a numerical experiment in the next subsection to highlight the effectiveness also of the BDF2. In addition to the algorithms just presented, we have also synthesized similar algorithms for gCQ based on the trapezoidal rule, both for forward and backward gCQ (see Algorithms \ref{alg1} and \ref{alg3}). It is worth noting that the forward scheme is used to compute a convolution like \eqref{equation1} when $g$ is known, while the backward scheme is used when $\phi$ is known. Similar algorithms for BDF1 gCQ can be found in \cite[Section 4]{LopezFernandezSauter2015b}, while for Runge-Kutta gCQ in \cite[Section 6]{LopezFernandezSauter2016} and \cite[Section 3]{LeitnerSchanz2021}.

\subsection{Quadrature aspects}

The idea is to compute step by step 
\begin{equation} \label{EQ22}
	\K\left(\partial_t^{\{\Delta_j\}}\right) g(t_n) = \frac{1}{2\pi \mi} \oint_{\mathcal{C}} \K_\rho(s) u_n(s) \dd s
\end{equation}
where $u_n$ is defined in \eqref{solTrap} and the complex integral is performed via suitable quadrature rules.

In order to utilize BDF1 with gCQ effectively, it is necessary to solve a quadrature problem. This issue has been successfully addressed in \cite{LopezFernandezSauter2015b}, with experimental results provided in \cite{LopezFernandezSauter2015a}. The circle contour is the optimal choice in this case, and it is parameterized using Jacobi elliptic functions to fully exploit the analyticity domain of the integrand in \eqref{EQ22}. It is worth noting that the poles of the integrand are located in the real segment.

We will briefly review the construction and the simple modification made in our case. As per the theoretical analysis presented in \cite{LopezFernandezSauter2015a,LopezFernandezSauter2015b}, we adopted $N_Q = N \log_2^2(N)$ quadrature points, where $N$ represents the number of time steps. The details and results of the aforementioned papers are summarized below. The integration points in the complex plane are
\begin{equation*}
	s_\ell := \gamma(\sigma_\ell), \quad w_\ell := \frac{4K(k^2)}{2 \pi \mi N_Q} \gamma'(\sigma_\ell), \quad \ell = 1,\ldots, N_Q
\end{equation*}
where the parameters $k$ and $\sigma_\ell$, depending on
\begin{equation} \label{qNOde}
	q:=\frac{M}{\Delta_{\min}}, \qquad M := R\max\left\{\Delta_{\max}^{-2},\Delta_{\min}^{-1}\right\}, \quad  \Delta_{\min} := \min\{\Delta_j\},
\end{equation}
are defined as
\begin{equation*}
	k := \frac{q- \sqrt{2q-1}}{q+\sqrt{2q-1}}, \quad \sigma_\ell := -K(k^2)+\left( \ell-\frac{1}{2} \right) \frac{4K(k^2)}{N_Q}
\end{equation*}
for $\ell = 1, \ldots, N_Q$.
The parameter $R$ is a constant depending on the underlined ODE solvers, precisely
\begin{equation*}
	R := 
	\begin{cases}
		1 & \text{BDF1} \\
		1.5 & \text{BDF2} \\
		2 & \text{trapezoidal rule}.
	\end{cases}
\end{equation*}
Finally, $K(k)$ is the complete elliptic integral of first kind
\begin{equation*}
	K(k) := \int_0^1 \frac{1}{\sqrt{(1-x^2)(1-k^2x^2)}} \dd x, \quad K'(k) = K(1-k),
\end{equation*} 
and the functions $\gamma$ is the parametrization of a circle centered in $M$ of radius $M$ (see \cite[Lemma 15]{LopezFernandezSauter2015b})
\begin{equation*}
\gamma(\sigma) := \frac{M}{q-1} \left( \sqrt{2q-1} \frac{k^{-1} + \text{sn}(\sigma | k^2)}{k^{-1} - \text{sn}(\sigma | k^2)} -1 \right), \quad
\gamma'(\sigma) = \frac{M\sqrt{2q-1}}{q-1} \frac{2 \ \text{cn}(\sigma | k^2) \text{dn}(\sigma | k^2)}{k(k^{-1}- \text{sn}(\sigma | k^2))}
\end{equation*}
where $\text{sn}, \text{dn}$ and $\text{cn}$ are the Jacobi elliptic functions whose evaluation have been performed in MATLAB by means of Driscoll's Schwarz-Christoffel Toolbox \cite{Driscoll2005}.

\begin{remark}
The only deviation from the nodes and weights proposed in \cite{LopezFernandezSauter2015b} is the introduction of a scaling factor, $R$, in \eqref{qNOde}. This parameter ensures that the complex poles of the integrands are suitably distanced from the contour of the circle with radius $M$ and center $M$ used for integration. For BDF1, the poles are $\{\Delta_j^{-1}\}$, and \cite{LopezFernandezSauter2015b} demonstrated that $R = 1$ is sufficient. In the case of BDF2, the poles are $\{A_j^{-1}\}$ as defined in \eqref{coefficients}. To extend the ideas put forth in \cite{LopezFernandezSauter2015b}, we set $R = 1.5$ for this case. In fact, as $\Delta_{\max}$ approaches $0$, we have $A_j^{-1} \to \frac{3}{2} \Delta_j^{-1} $. Finally, for the trapezoidal rule, the poles are $\{2\Delta_j^{-1}\}$, and we have selected $R=2$.
\end{remark}
%%

%%%%%%%%%%%%%%%%%%%%%%%%%
\subsection{Numerical results}
We consider the following one-dimensional example, already performed in \cite{LopezFernandezSauter2015b, LopezFernandezSauter2016}: find $g$ such $\K(\partial_t)g = \phi$ with 
\begin{equation} \label{Kphi}
\K(s) := \frac{1-e^{-2s}}{2s}, \quad \text{and} \quad \phi(t) := t^{\nicefrac{5}{2}} e^{-t}.
\end{equation}
The exact solution to this problem is given by
\begin{equation*}
	g(t) := 2\sum_{k=0}^{\lfloor \frac{t}{2} \rfloor} \phi'(t-2k)
\end{equation*}
We approximate $g(t)$ for $t \in [0, 1]$ by applying Algorithms \ref{alg3} and \ref{alg4}. Note that $\K^{-1}$ satisfies \eqref{assump} with $\mu=1$. The right hand side $\phi$ satisfies $\phi^{(j)}(0) = 0$ for $j=0, 1, 2$ but is not three times differentiable at $t=0$. This lack of regularity suggests to use a time grid which is algebraically graded towards the origin. We choose a graded mesh with points
\begin{equation} \label{grids}
	t_j = \left( \frac{j}{N} \right)^{\alpha}, \quad j = 0,\ldots N
\end{equation}
Figure \ref{fig3} shows that the convergence rate is $\mathcal{O}(\Delta^2)$ for the quadratic graded mesh ($\alpha=2$) and about $\mathcal{O}(\Delta^{1.5})$ for the uniform mesh ($\alpha=1$). For this example, we have $\mu = 1$, which implies that the minimal integer $\rho > \mu + 3 = 5$. Note that, however, we have used $\rho=0$ in our computations. It remains an open problem whether there exist examples where a larger value of $\rho$ is necessary for variable steps than for uniform steps, or whether our theory yields a suboptimal estimate in terms of this parameter.

\begin{figure}[h!]
  \centering
  \begin{minipage}[b]{0.275\textwidth}
    \includegraphics[width=\textwidth]{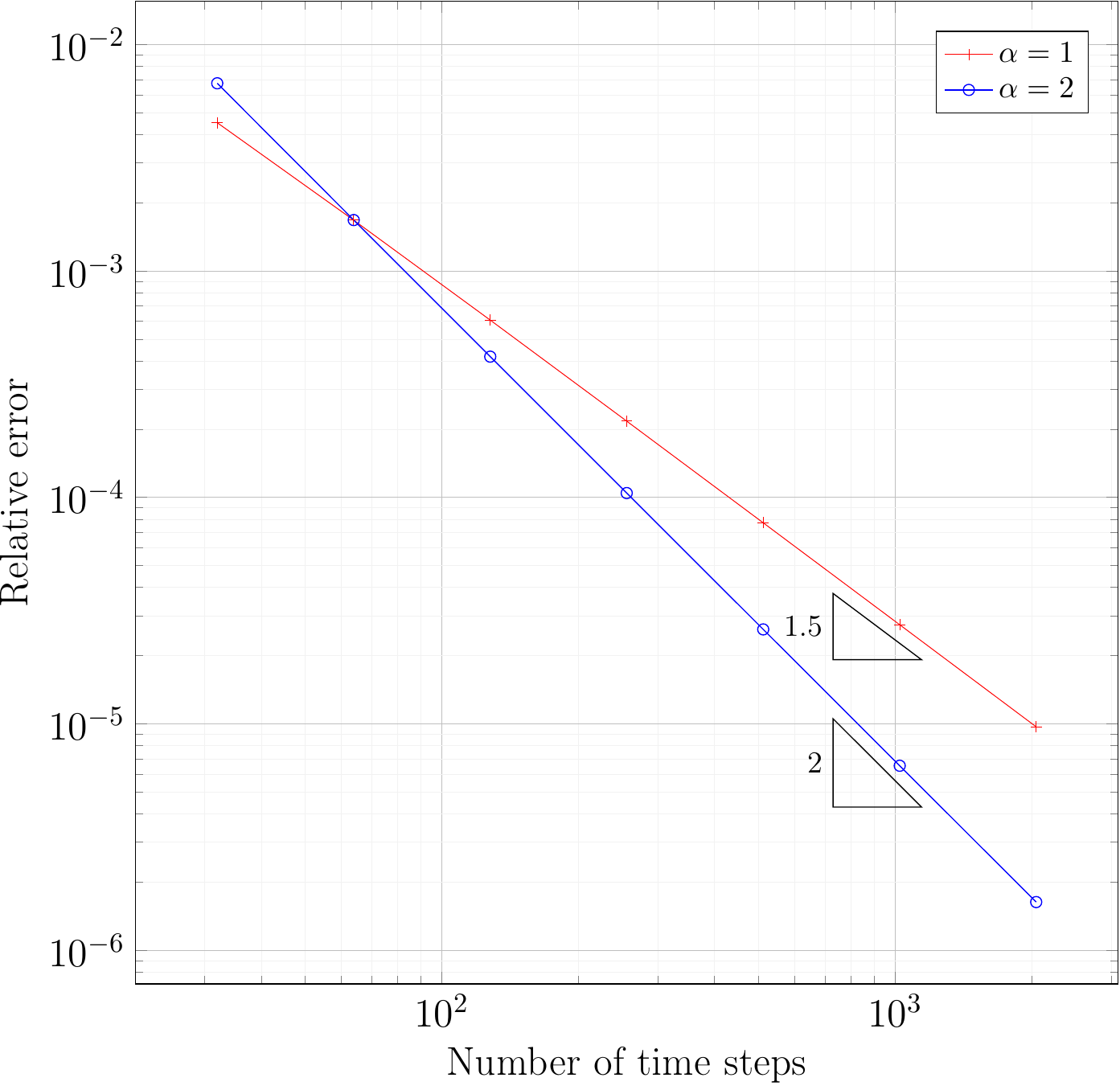}
  \end{minipage}
  \hspace{1cm}
  \begin{minipage}[b]{0.28\textwidth}
    \includegraphics[width=\textwidth]{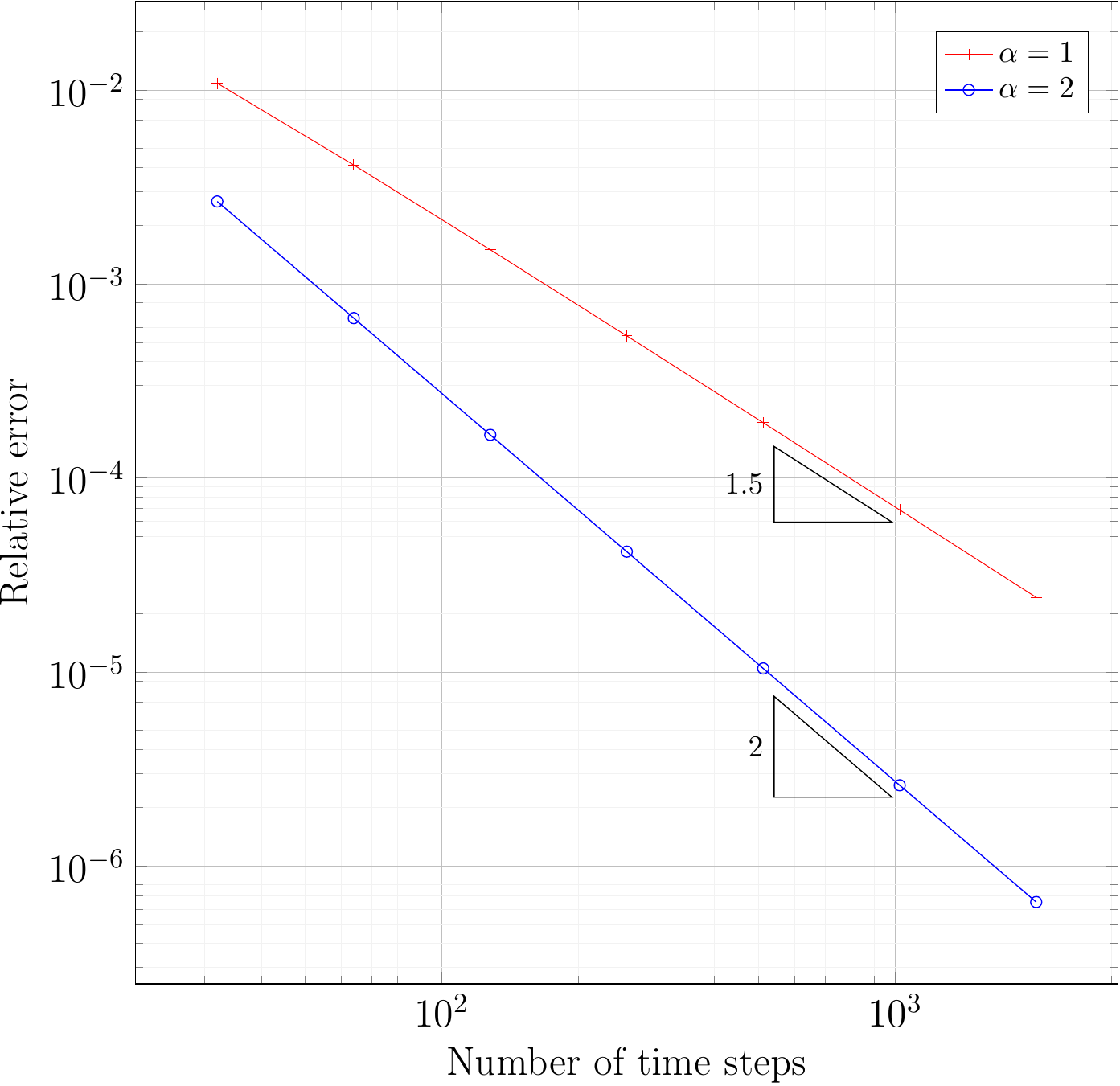}
  \end{minipage}
\caption{Error with respect to the number of steps for the data in \eqref{Kphi} for different grids \eqref{grids}, obtained with gCQ based on the trapezoidal rule (left) and on BDF2 (right).} \label{fig3}
\end{figure}

\begin{figure}[h!]
  \centering
  \begin{minipage}[b]{0.25\textwidth}
    \includegraphics[width=\textwidth]{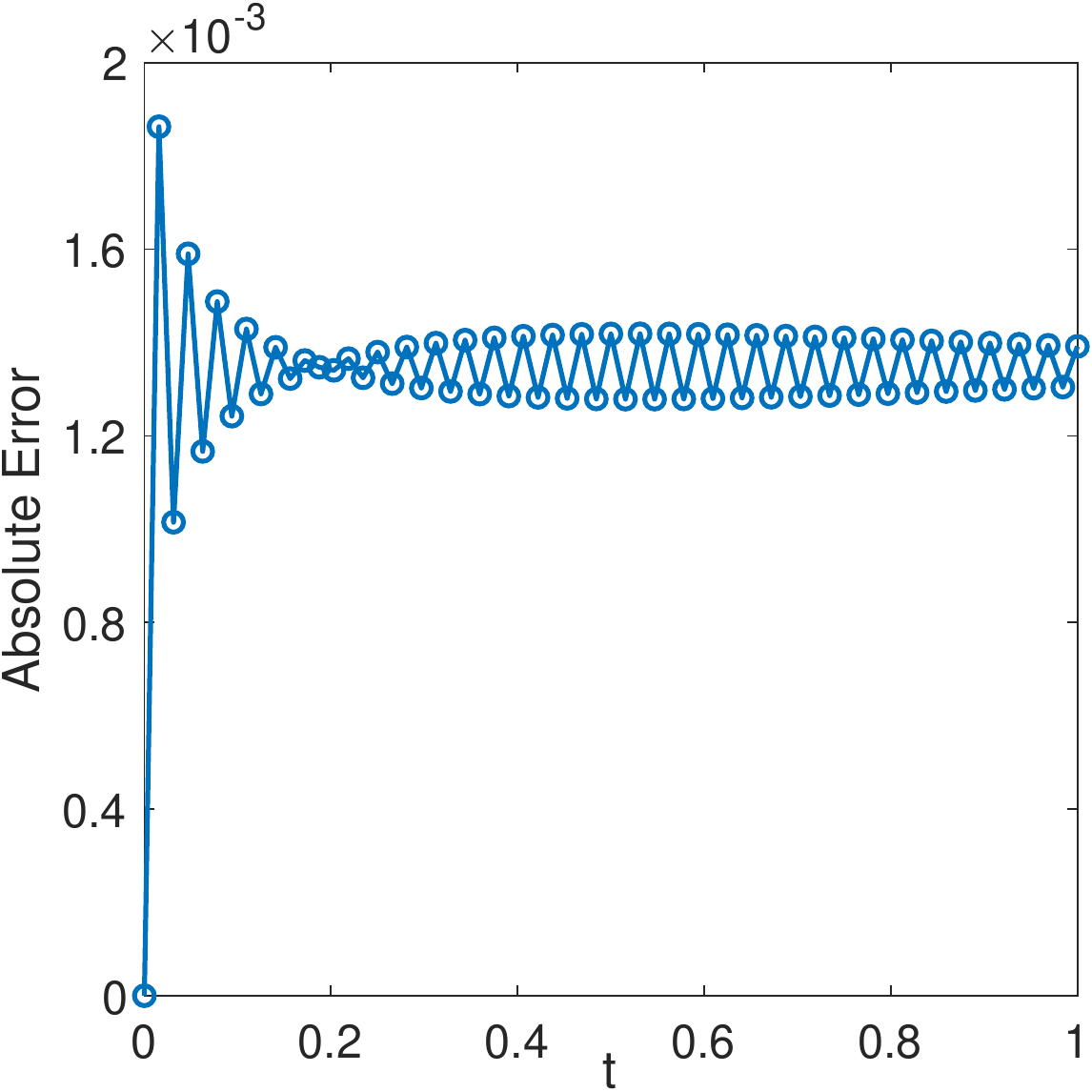}
  \end{minipage}
\hspace{2cm}
  \begin{minipage}[b]{0.25\textwidth}
    \includegraphics[width=\textwidth]{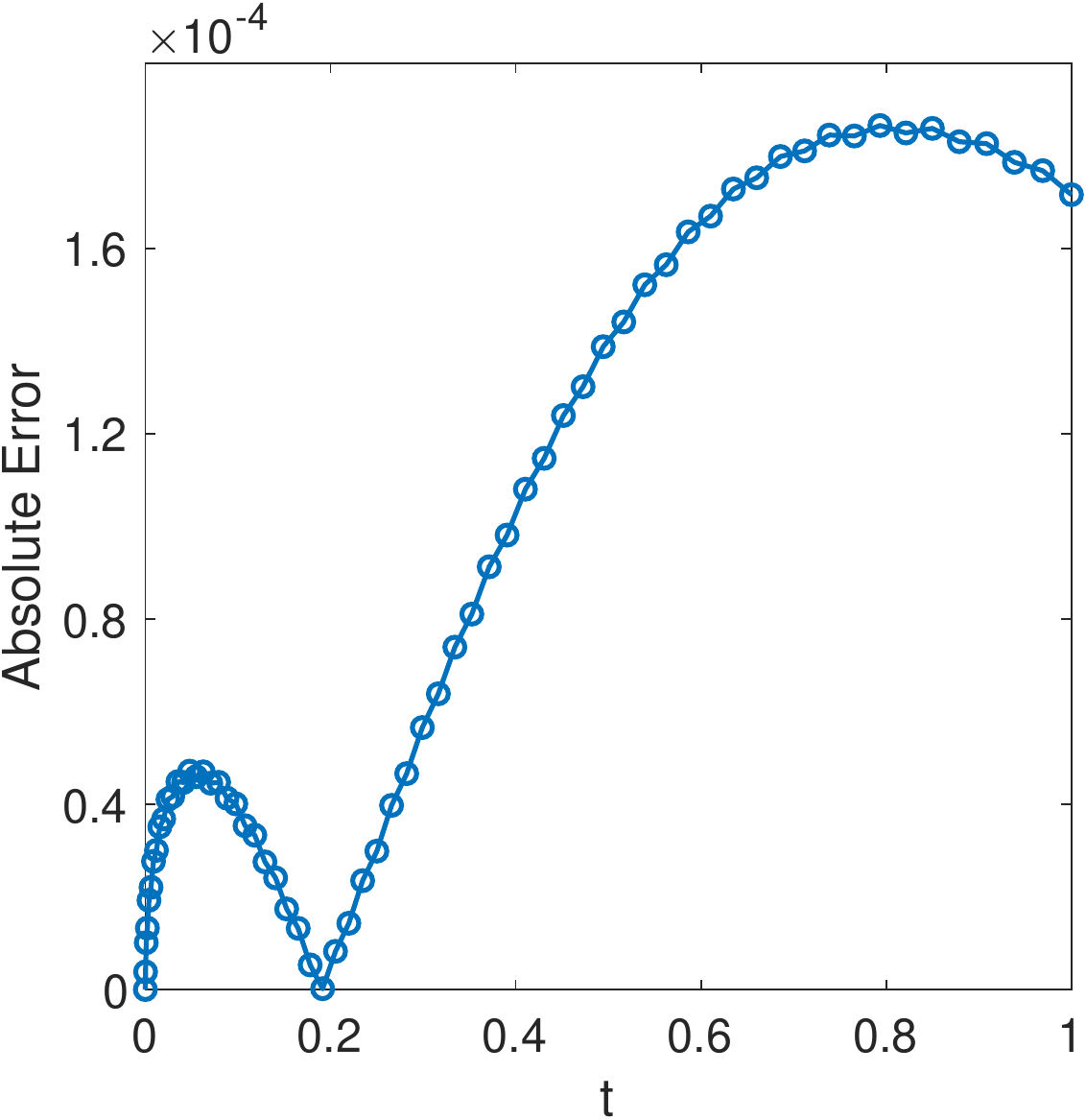}
  \end{minipage}
  \caption{Absolute error in the approximation with the trapezoidal rule for the data in \eqref{Kphi} with $N=64$ time steps: with uniform time steps, i.e. $\alpha=1$, in the left, and with quadratically graded time steps, i.e. $\alpha=2$, in the right.} \label{fig1}
\end{figure}

\begin{figure}[h!]
  \centering
  \begin{minipage}[b]{0.25\textwidth}
    \includegraphics[width=\textwidth]{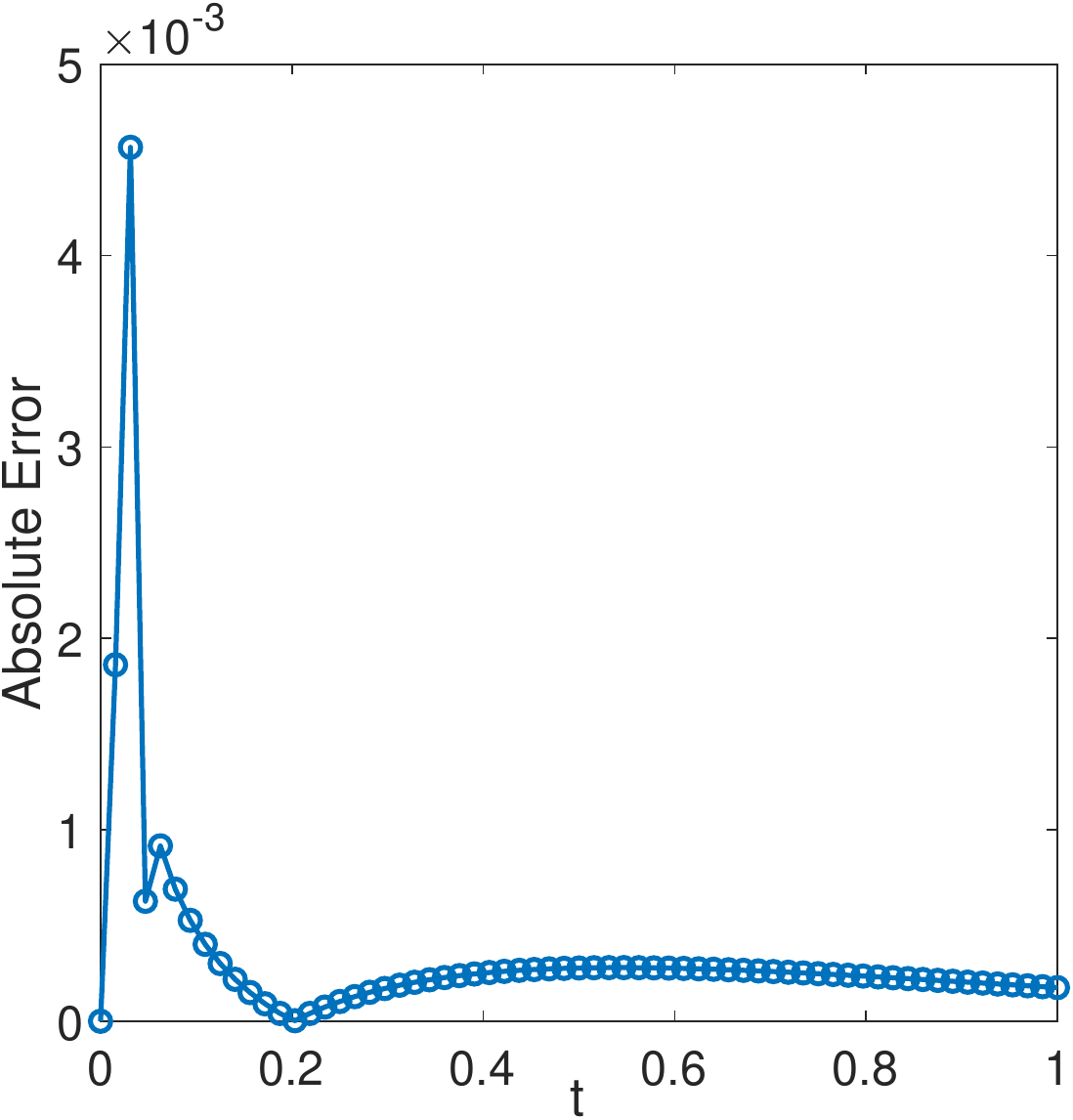}
  \end{minipage}
\hspace{2cm}
  \begin{minipage}[b]{0.25\textwidth}
    \includegraphics[width=\textwidth]{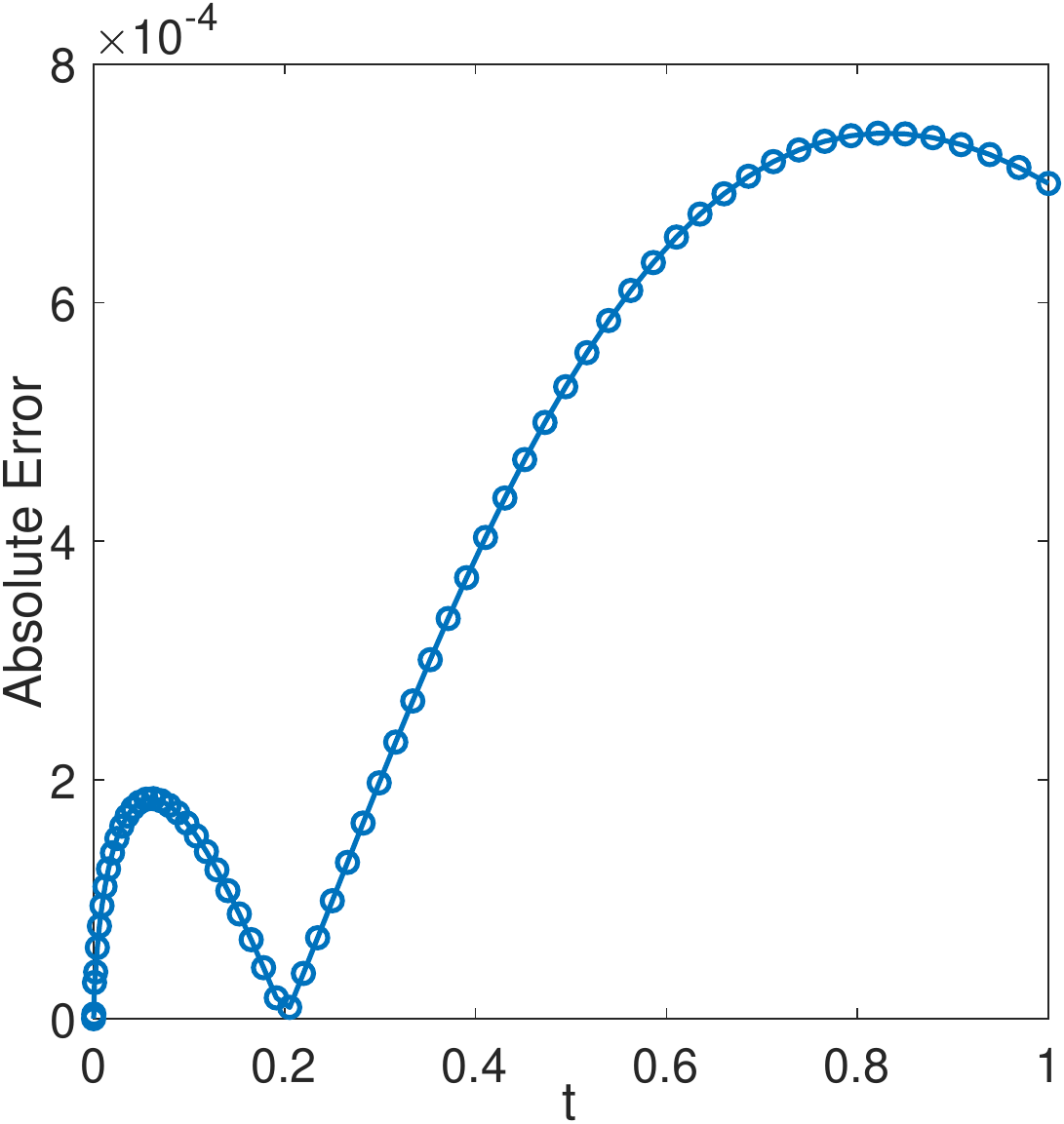}
  \end{minipage}
  \caption{Absolute error in the approximation with the BDF2 for the data in \eqref{Kphi} with $N=64$ time steps: with uniform time steps, i.e. $\alpha=1$, in the left, and with quadratically graded time steps, i.e. $\alpha=2$, in the right.} \label{fig2}
\end{figure}
%%

%%%%%%%%%%%%%%%%%%%%%%%%%%%
\begin{algorithm}
\caption{\textit{Forward} gCQ with contour quadrature based on trapezoidal rule}\label{gCQtrap}
\label{alg1}
\begin{algorithmic}
\State\textbf{Initialization} Generate $\K_\rho(s_{\ell})$ for all contour quadrature nodes $s_\ell$, $\ell = 1,\ldots,N_Q$. Compute $\phi_1$ from 
\begin{equation*}
	\phi_1  = \K_\rho\left(\frac{2}{\Delta_1} \right) g^{(\rho)}(t_1).
\end{equation*}
\State Set $u_0(s) \equiv 0$.
\For {$n = 2,\ldots,N$}
\vspace{0.1cm}
\State \textbf{1. Trapezoidal step:} apply a step of the trapezoidal rule and compute
\begin{equation*}
u_{n-1}(s_\ell) = u_{n-2}(s_\ell) \frac{2+\Delta_{n-1}s_\ell}{2-\Delta_{n-1}s_\ell} +  \left(g^{(\rho)}(t_{n-2}) + g^{(\rho)}(t_{n-1})\right) \frac{\Delta_{n-1}}{2-\Delta_{n-1} s_\ell},
\end{equation*}
\State for all contour quadrature nodes $\ell=1,\ldots,N_Q$.
\State \textbf{2. Compute $\phi_n$:} if $\Delta_n$ is a new time step, then, generate $\K_\rho\left(\frac{2}{\Delta_n}\right)$; otherwise this operator was already
\State generated in a previous step. Compute $\phi_n$ from
\begin{equation*}
	 \phi_n = \sum_{\ell=1}^{N_Q} w_\ell \K_{\rho}(s_\ell) u_{n-1}(s_\ell) \frac{2+\Delta_n s_\ell}{2- \Delta_n s_\ell} +  \K_\rho\left(\frac{2}{\Delta_n} \right) \left(g^{(\rho)}(t_{n-1}) + g^{(\rho)}(t_{n})\right).
\end{equation*}
\EndFor
\end{algorithmic}
\end{algorithm}
%%%%%%%%%%%%%%%%%%

\begin{algorithm} 
\caption{\textit{Forward} gCQ with contour quadrature based on BDF2}
\label{alg2}
\begin{algorithmic}
\State\textbf{Initialization} Generate $\K_\rho(s_{\ell})$ for all contour quadrature nodes $s_\ell$, $\ell = 1,\ldots,N_Q$. Compute $\phi_1$ from 
\begin{equation*}
	\phi_1 = \K_\rho\left(\frac{1}{A_1} \right) g^{(\rho)}(t_1).
\end{equation*}
Set $u_0(s) \equiv u_{-1}(s) \equiv 0$.
\For{$n=2,\ldots,N$}
\State \textbf{1. BDF2 step:} apply a step of the BDF2 and compute
\begin{align*}
u_{n-1}(s_\ell) =  u_{n-2}(s_\ell) \frac{B_{n-1}}{1-A_{n-1} s_\ell} - u_{n-3}(s_\ell)\frac{C_{n-1}}{1-A_{n-1}s_\ell}  + g^{(\rho)}(t_{n-1})\frac{A_{n-1}}{1-A_{n-1} s_\ell}
\end{align*}
\State with coefficients \eqref{coefficients}, for all contour quadrature nodes $\ell=1,\ldots,N_Q$.
\State \textbf{2. Compute $\phi_n$:} if $A_n$ is different from the previous coefficients, then generate $\K_\rho\left(\frac{1}{A_n}\right)$; otherwise
\State this operator was already generated in a previous step. Compute $\phi_n$ from
\begin{equation*}
	\phi_n = \sum_{\ell=1}^{N_Q} w_\ell \K_{\rho}(s_\ell) \left( u_{n-1}(s_\ell) \frac{B_n}{1- A_n s_\ell} - u_{n-2}(s_\ell) \frac{C_n}{1- A_n s_\ell} \right) + \K_\rho\left(\frac{1}{A_n}\right) g^{(\rho)}(t_n).
\end{equation*}
\EndFor
\end{algorithmic}
\end{algorithm}
%%%%%%%

\begin{algorithm} 
\caption{\textit{Backward} gCQ with contour quadrature based on trapezoidal rule}
\label{alg3}
\begin{algorithmic}
\State\textbf{Initialization} Generate $\K_{-\rho}(s_{\ell})$ for all contour quadrature nodes $s_\ell$, $\ell = 1,\ldots,N_Q$. Compute $g_1$ from 
\begin{equation*}
	\K_{-\rho}\left(\frac{2}{\Delta_1} \right) g_1 = \phi^{(\rho)}(t_1).
\end{equation*}
\State Set $u_0(s) \equiv 0$.
\For {$n = 2,\ldots,N$}
\vspace{0.1cm}
\State \textbf{1. Trapezoidal step:} apply a step of the trapezoidal rule and compute
\begin{equation*}
u_{n-1}(s_\ell) = u_{n-2}(s_\ell)  \frac{2+\Delta_{n-1}s_\ell}{2-\Delta_{n-1}s_\ell} +  \left(g_{n-2} + g_{n-1}\right) \frac{\Delta_{n-1}}{2-\Delta_{n-1} s_\ell},
\end{equation*}
\State for all contour quadrature nodes $\ell=1,\ldots,N_Q$.
\State \textbf{2. Generate linear system:} if $\Delta_n$ is a new time step, then, generate $\K_{-\rho}\left(\frac{2}{\Delta_n}\right)$; otherwise this operator 
\State was already generated in a previous step. Update the right-hand side
\begin{equation*}
	r_n := \phi^{(\rho)}(t_n) -  \sum_{\ell=1}^{N_Q} w_\ell \K_{-\rho}(s_\ell) u_{n-1}(s_\ell) \frac{2+\Delta_n s_\ell}{2- \Delta_n s_\ell} - \K_{-\rho}\left(\frac{2}{\Delta_{n-1}} \right) g_{n-1}.
\end{equation*}
\State \textbf{3. Linear Solve:} solve the linear system
\begin{equation*}
	\K_{-\rho}\left(\frac{2}{\Delta_n} \right) g_n = r_n.
\end{equation*}
\EndFor
\end{algorithmic}
\end{algorithm}
%%%%%%%%%%%%%%%%%%%%%%%%%
\begin{algorithm}
\caption{\textit{Backward} gCQ with contour quadrature based on BDF2}
\label{alg4}
\begin{algorithmic}
\State\textbf{Initialization} Generate $\K_{-\rho}(z_{\ell})$ for all contour quadrature nodes $s_\ell$, $\ell = 1,\ldots,N_Q$. Compute $g_1$ from 
\begin{equation*}
	\K_{-\rho}\left(\frac{1}{A_1} \right) g_1 = \phi^{(\rho)}(t_1).
\end{equation*}
Set $u_0(s) \equiv u_{-1}(s) \equiv 0$.
\For{$n=2,\ldots,N$}
\State \textbf{1. BDF2 step:} apply a step of the trapezoidal rule and compute
\begin{align*}
u_{n-1}(s_\ell) = & u_{n-2}(s_\ell) \frac{B_{n-1}}{1-A_{n-1} s_\ell} - u_{n-3}(s_\ell) \frac{C_{n-1}}{1-A_{n-1}s_\ell} + g_{n-1} \frac{A_{n-1}}{1-A_{n-1} s_\ell}
\end{align*}
\State with coefficients \eqref{coefficients}, for all contour quadrature nodes $\ell=1,\ldots,N_Q$.
\State \textbf{2. Generate linear system:}  if $A_n$ is different from the previous coefficients, then generate $\K_\rho\left(\frac{1}{A_n}\right)$; otherwise,
\State this operator was already generated in a previous step. Update the right-hand side
\begin{equation*}
	r_n := \phi^{(\rho)}(t_n) - \sum_{\ell=1}^{N_Q} w_\ell \K_{-\rho}(s_\ell) \left( u_{n-1}(s_\ell) \frac{B_n}{1- A_n s_\ell} - u_{n-2}(s_\ell) \frac{C_n}{1- A_n s_\ell} \right).
\end{equation*}
\State \textbf{3. Linear Solve:} solve the linear system
\begin{equation*}
	\K_{-\rho}\left(\frac{1}{A_n} \right) g_n = r_n.
\end{equation*}
\EndFor
\end{algorithmic}
\end{algorithm}
%%%%%%%
%%%%%%%
\section{Conclusion}
We present an improved approach for solving one-sided convolution equations: the gCQ method with variable time stepping based on the trapezoidal rule, which we develop and analyze in this paper. This method builds on the original CQ method, which transforms the continuous equation to the Laplace domain and characterizes the transformed solution as an ODE. In contrast to the CQ method, we introduce variable time stepping for the solution of the ODE, resulting in the gCQ method with improved accuracy and efficiency. Specifically, we utilize the trapezoidal rule for the time stepping in the gCQ method.

To analyze the gCQ method, we develop a theory based on a new formula of quadrature integral and a pointwise error estimate of the weights. The gCQ method is also implemented in a stable algorithmic version based on both the trapezoidal and the BDF2 rules, and we report the results of numerical experiments illustrating the advantages of variable time stepping for non-smooth data. It is worth noting that constructing a quadrature on an appropriate contour in the integral formula \eqref{EQ22} is used for stable computation, as seen in \cite{LopezFernandezSauter2015b}. However, the fast FFT algorithms for the uniform time-step CQ are not available. A study on the stability and convergence of the BDF2 method will be conducted in a future work

\section*{Acknowledgments}
The second author is members of the GNCS group (\textit{Gruppo Nazionale per il Calcolo Scientifico}) of INdAM (\textit{Istituto Nazionale di Alta Matematica ``F. Severi''}). The second author was partially supported from the MIUR grant \textit{Dipartimenti di Eccellenza 2018-2022} (E11G18000350001) of the Italian Ministry for University and Research.

\bibliographystyle{plain}
\bibliography{bibtex_num}

\end{document}